\documentclass[a4paper]{amsart}
\usepackage[active]{srcltx}
\usepackage[all]{xy}
\usepackage{subfig}
\usepackage{hyperref}
\usepackage{mathrsfs}

\usepackage{esint}
\usepackage{amsmath}
\usepackage{amssymb,latexsym}
\usepackage{mathrsfs}
\usepackage{graphics}
\usepackage{latexsym}
\usepackage{psfrag}
\usepackage{import}
\usepackage{verbatim}
\usepackage{graphicx}
\usepackage[usenames]{color}
\usepackage{pifont,marvosym}

\theoremstyle{plain}
\newtheorem{lemma}{Lemma}[section]
\newtheorem{theorem}[lemma]{Theorem}
\newtheorem{proposition}[lemma]{Proposition}
\newtheorem{corollary}[lemma]{Corollary}

\theoremstyle{definition}

\newtheorem{definition}[lemma]{Definition}
\newtheorem{remark}[lemma]{Remark}

\numberwithin{equation}{section}

\newcommand{\dom}{{\textrm{Dom\,}}}

\newcommand{\R}{\mathbb{R}}
\newcommand{\N}{\mathbb{N}}

\newcommand{\dive}{\text{\rm div}}

\newcommand{\supp}{\text{\rm supp}}

\newcommand{\Lip}{\mathrm{Lip}}

\newcommand{\diam}{{\rm{diam\,}}}

\newcommand{\ve}{\varepsilon}

\newcommand{\erre}{\mathbb{R}}
\newcommand{\cI}{\mathcal{I}}

\newcommand{\f}{\varphi}

\renewcommand{\r}{\varrho}

\renewcommand{\L}{\mathcal{L}}
\newcommand{\RCD}{\mathsf{RCD}}

\newcommand{\CD}{\mathsf{CD}}
\newcommand{\Geo}{{\rm Geo}}

\newcommand{\mm}{\mathfrak m}
\newcommand{\qq}{\mathfrak q}
\newcommand{\ee}{{\rm e}}

\newcommand{\sfd}{\mathsf d}
\newcommand{\Opt}{\mathrm{OptGeo}}
\newcommand{\bigslant}[2]{{\raisebox{.2em}{$#1$}\left/\raisebox{-.2em}{$#2$}\right.}}

\begin{document}

\title[Sharp geometric and functional inequalities for $\CD^*(K,N)$ spaces]
{Sharp geometric and functional inequalities \\ in metric measure spaces with lower ricci curvature bounds}
\author{Fabio Cavalletti} 
\address{Universit\`a di Pavia}
\email{fabio.cavalletti@unipv.it}

\author{Andrea Mondino}
\address{ETH - Zurich}
\email{andrea.mondino@math.uzh.ch}


\bibliographystyle{plain}

\begin{abstract}
For metric measure spaces satisfying the reduced curvature-dimension condition $\CD^*(K,N)$ we prove a series of sharp functional inequalities under the 
additional assumption of essentially non-branching. Examples of spaces entering this framework are (weighted) Riemannian manifolds satisfying lower Ricci curvature bounds and their measured Gromov Hausdorff limits, Alexandrov spaces satisfying lower curvature bounds and more generally $\RCD^*(K,N)$-spaces, Finsler manifolds endowed with a strongly convex norm and satisfying lower Ricci curvature bounds.

In particular we prove the  Brunn-Minkowski inequality, the $p$-spectral gap (or equivalently the $p$-Poincar\'e inequality) for any $p\in [1,\infty)$,  the log-Sobolev inequality, the Talagrand inequality and finally  the Sobolev inequality. 

All the  results are proved in a sharp form involving an upper bound on the diameter of the space; all our inequalities for essentially non-branching $\CD^*(K,N)$ spaces take the same form as the corresponding sharp ones known for a weighted  Riemannian manifold satisfying the curvature-dimension condition $\CD(K,N)$ in the sense of Bakry-\'Emery. In this sense inequalities are sharp.   We also discuss the rigidity and almost rigidity statements associated to the $p$-spectral gap. 

Finally let us mention that for essentially non-branching metric measure spaces, the local curvature-dimension condition $\CD_{loc}(K,N)$ is equivalent to the reduced curvature-dimension condition $\CD^*(K,N)$. 
Therefore we also have shown that the  sharp  Brunn-Minkowski inequality in the \emph{global} form can be deduced from the \emph{local} curvature-dimension condition, providing a step towards 
(the long-standing problem of) globalization for the curvature-dimension condition $\CD(K,N)$.
 
To conclude, some of the results can be seen as answers to open problems proposed in the Optimal Transport book of Villani \cite{Vil}. 

\end{abstract}

\maketitle


\section{Introduction}

The theory of metric measure spaces satisfying a synthetic version of lower curvature and upper dimension bounds is nowadays 
a rich and well-established theory; nevertheless some important functional and geometric inequalities are in some cases 
still not proven and in others not proven in a sharp form. 
The scope of this note is to generalize several functional inequalities known for Riemannian manifolds satisfying a lower bound on the Ricci curvature 
to the more general case of metric measure spaces satisfying the so-called curvature-dimension condition $\CD(K,N)$
as defined by Lott-Villani \cite{lottvillani:metric} and Sturm \cite{sturm:I, sturm:II}. 
More precisely our results will hold under the \emph{reduced} curvature dimension condition $\CD^*(K,N)$ introduced by Bacher-Sturm \cite{BS10}  
(which is, a priori, a weaker assumption than the classic  $\CD(K,N)$) coupled with an essentially non-branching assumption on geodesics. 
We refer  to Section \ref{Ss:geom} for the precise definitions; here let us recall that remarkable examples of essentially 
non-branching $\CD^*(K,N)$ spaces are  (weighted) Riemannian manifolds satisfying lower Ricci curvature bounds 
and their measured Gromov Hausdorff limits, Alexandrov spaces satisfying lower curvature bounds 
and more generally $\RCD^*(K,N)$-spaces, Finsler manifolds endowed with a strongly convex norm and satisfying lower Ricci curvature bounds.

\begin{remark}\label{R:intro}
To avoid technicalities in the introduction, all the results will be stated for $N>1$; nevertheless everything holds (and will be proved in the paper)  
also for  $N=1$, but in this case $\CD^{*}(K,N)$ has to be replaced by $\CD_{loc}(K,N)$. The two conditions are equivalent for $N>1$ 
and for $N=1, K \geq 0$, but in case $N=1, K< 0$ the $\CD_{loc}(K,N)$ condition is  strictly stronger (see Section  \ref{Ss:geom} for more details).
\end{remark}

Before committing a paragraph to each of the functional inequalities we will consider in this note, 
we underline that most of the proofs contained in this note are based on $L^{1}$ optimal transportation theory and in particular 
on  one-dimensional localization.  This  technique, having its roots  in a work of   Payne-Weinberger \cite{PW} and developed by Gromov-Milman \cite{GrMi}, Lov\'asz-Simonovits \cite{LoSi} and Kannan-Lov\'asz-Simonovits \cite{KaLoSi},  consists in reducing an $n$-dimensional problem to a one dimensional one  via tools of convex geometry. Recently Klartag \cite{klartag}  found an $L^1$-optimal transportation approach leading to a  generalization of these ideas to  Riemannian manifolds;   the authors  \cite{CM1}, via a careful analysis avoiding any smoothness assumption,  generalized this approach to metric measure spaces.

It is also convenient to introduce here the family of one-dimensional measures that will be used several times for  comparison: 
$$
\mathcal{F}_{K,N,D} : = \left\{ \mu \in \mathcal{P}(\R) \colon \supp(\mu) \subset [0,D], 
\,  \mu = h_{\mu} \cdot \mathcal{L}^{1}, \,h_{\mu} \in C^{2}((0,D)), \, (\R, |\cdot |, \mu) \in \CD(K,N)  \right\}, 
$$
where $(\R, |\cdot |, \mu) \in \CD(K,N)$ stands for: the metric measure space $(\R, |\cdot |, \mu)$ verifies $\CD(K,N)$ or equivalently 
$$
\left( h_{\mu}^{\frac{1}{N-1}} \right)^{''} + \frac{K}{N-1} \, h_{\mu}^{\frac{1}{N-1}} \leq 0.
$$

\subsection{Brunn-Minkowski inequality}

The celebrated Brunn-Minkowski inequality estimates from below the measure of the 
$t$-intermediate points between two given subsets $A_{0}$ and $A_{1}$ of $X$, for $t\in [0,1]$. 
For metric measure spaces satisfying the reduced curvature-dimension condition $\CD^{*}(K,N)$ 
(see Section \ref{Ss:geom} for a brief account of different versions of the curvature-dimension condition)
almost by definition for any $A_{0}, A_{1} \subset X$
\begin{equation}\label{E:BMreduced}
\mm(A_{t})^{1/N}  \geq \sigma_{K,N}^{(1-t)}(\theta)\, \mm(A_{0})^{1/N} + \sigma_{K,N}^{(t)}(\theta) \, \mm(A_{1})^{1/N},
\end{equation}
where $A_{t}$ is the set of $t$-intermediate points between $A_{0}$ and $A_{1}$, that is 
$$
A_{t} = \ee_{t} \Big(  \{ \gamma \in \Geo(X) \colon \gamma_{0} \in A_{0}, \gamma_{1} \in A_{1}   \}  \Big),
$$
(see Section \ref{S:pre} for the definition of $\ee$) $\theta$ is the minimal/maximal length of geodesics from $A_{0}$ to $A_{1}$: 
$$
\theta:= 
\begin{cases}
\inf_{(x_{0},x_{1}) \in A_{0} \times A_{1}} \sfd(x_{0},x_{1}), & \textrm{if } K \geq 0, \\
\sup_{(x_{0},x_{1}) \in A_{0} \times A_{1}} \sfd(x_{0},x_{1}), & \textrm{if } K < 0,
\end{cases}
$$
and $\sigma_{K,N}^{(t)}(\theta)$ is defined in \eqref{E:tau}. Nevertheless \eqref{E:BMreduced} is not sharp. 
Indeed if $(X,\sfd,\mm)$ is a weighted Riemannian manifold satisfying $\CD^{*}(K,N)$, then \eqref{E:BMreduced} holds but with better interpolation coefficients, 
that is with $\tau_{K,N}^{(t)}(\theta), \tau_{K,N}^{(1-t)}(\theta)$ replacing $\sigma_{K,N}^{(t)}(\theta)$ and $\sigma_{K,N}^{(1-t)}(\theta)$, respectively. 
Indeed for  a weighted Riemannian manifold the two (a priori) different definitions of $\CD^{*}(K,N)$ and $\CD(K,N)$ 
coincide and then again almost by definition \cite{sturm:II} one can obtain the improved (and sharp) Brunn-Minkowski inequality (let us mention that a direct proof of the Brunn-Minkowski inequality in the smooth setting was done earlier by  Cordero-Erausquin, McCann and Schmuckenschl\"ager  \cite{CMS}).

A first main result of this paper is to establish the  sharp inequality for essentially non-branching  $\CD^*(K,N)$ metric measure spaces.  

\begin{theorem}[Theorem \ref{T:BM}]\label{T:introBM}
Let $(X,\sfd,\mm)$ with $\mm(X)<\infty$ verify $\CD^{*}(K,N)$ for some $K,N \in \R$ and $N \in (1, \infty)$. Assume moreover $(X,\sfd,\mm)$ to be essentially non-branching.
Then it satisfies the following sharp Brunn-Minkowski inequality: 

for any $A_{0}, A_{1} \subset X$
$$
\mm(A_{t})^{1/N}  \geq \tau_{K,N}^{(1-t)}(\theta)\, \mm(A_{0})^{1/N} + \tau_{K,N}^{(t)}(\theta) \,\mm(A_{1})^{1/N},
$$
where $A_{t}$ is the set of $t$-intermediate points between $A_{0}$ and $A_{1}$
and $\theta$ the minimal/maximal length of geodesics from $A_{0}$ to $A_{1}$. 
\end{theorem}

\begin{remark}
The remarkable feature of Theorem \ref{T:introBM} is  that the sharp Brunn-Minkowski inequality in the \emph{global} form can be deduced from the \emph{local} curvature-dimension condition, providing a step towards 
(the long-standing problem of) globalization for the curvature-dimension condition $\CD(K,N)$. For an account and for partial results  about  this problem we refer to \cite{AMSLocToGlob, BS10, cava:decomposition,  cavasturm:MCP, RajalaLocToGlob, Vil}. 
\end{remark}

\medskip

\subsection{$p$-Spectral gap}

In the smooth setting, a spectral gap inequality establishes a bound from below on the first eigenvalue of the Laplacian. 
More generally, for any $p \in  (1,\infty)$ one can define the positive real number $\lambda^{1,p}_{(X,\sfd,\mm)}$ as follows
$$
\lambda_{(X,\sfd,\mm)}^{1,p} 
: = \inf \left\{ \frac{\int_{X} |\nabla f |^{p} \, \mm}{\int_{X} |f|^{p}\,\mm} \colon f \in \Lip(X) \cap L^{p}(X,\mm), \ f \neq 0, \ \int_{X} f|f|^{p-2} \, \mm  = 0\right\},
$$
where $|\nabla f|$ is the slope (also called local Lipshitz constant) of the Lipschitz function $f$. The name is motivated by the fact that in case $(X,\sfd,\mm)$ is the m.m.s corresponding to a smooth compact Riemannian manifold then 
$\lambda^{1,p}_{(X,\sfd,\mm)}$ coincides with the first  positive eigenvalue of the problem 
$$
\Delta_p f= \lambda |f|^{p-2} f, 
$$
on $(X,\sfd,\mm)$, where $\Delta_p f:= -\dive (|\nabla f|^{p-2} \nabla f)$ is the 
so called $p$-Laplacian.

We now state  the main theorem of this paper on $p$-spectral gap inequality. 

\begin{theorem}[Theorem \ref{T:spectralgap}]\label{T:introSpectral}
Let $(X,\sfd,\mm)$ be a metric measure space satisfying $\CD^{*}(K,N)$ for some $K,N \in \R$ and $N \in (1, \infty)$ and assume moreover it is essentially non-branching. 
Let $D \in (0,\infty)$ be the diameter of $X$. 

Then for any $p \in (1,\infty)$ it holds 
$$
\lambda^{1,p}_{(X,\sfd,\mm)} \ \geq \ \lambda^{1,p}_{K,N,D}, 
$$
where $\lambda^{1,p}_{K,N,D}$ is defined by 
$$
\lambda_{K,N,D}^{1,p}: =  \inf_{\mu \in \mathcal{F}_{K,N,D}} \inf \left\{ \frac{\int_{\R} |u'|^{p}\, \mu }{\int_{\R} |u|^{p}\, \mu } : 
u \in \Lip(\R) \cap L^{p}(\mu), \, \int_{\R} u |u|^{p-2}\mu =0,\, u\neq 0  \right\}.
$$

\noindent
In other terms for any Lipschitz function $f \in L^{p}(X,\mm)$ with $\int_{X} f |f|^{p-2} \, \mm(dx) = 0$ it holds
$$
\lambda^{1,p}_{K,N,D} \, \int_{X} |f(x)|^{p}\, \mm(dx) \leq \int_{X} |\nabla f|^{p} (x) \, \mm(dx).
$$
\end{theorem}
For more about   the quantity $\lambda^{1,p}_{K,N,D}$ the reader is referred to Section \ref{Ss:modelspectral} where the model spaces are discussed in detail.
From the last formulation of the statement, it is  clear that the sharp  $p$-spectral gap above is equivalent to a sharp $p$-Poincar\'e inequality.

Let us now give a brief (and incomplete) account on the huge literature about the spectral gap. 
\\When the ambient metric measure space is a smooth Riemannian manifold equipped with the volume measure, 
the study of the first eigenvalue of the Laplace-Beltrami operator has a long history going back to Lichnerowicz \cite{Lich}, Cheeger \cite{Chee},
Li-Yau \cite{LY79}, etc. For an overview the reader can consult for instance the book by Chavel \cite{ChavEigen}, the  survey by Ledoux \cite{Led}, or  Chapter 3 in Shoen-Yau's book \cite{SYLect}, and references therein. 

We mention that the estimate of Theorem \ref{T:introSpectral} in the case $p=2$ started with 
Payne-Weinberger \cite{PW} for convex domains in $\R^n$ where diameter-improved spectral gap inequality for the 
Laplace operator was originally proved.
Later it was generalized to Riemannian manifolds with non-negative Ricci curvature by Yang-Zhong \cite{YZ}, 
and by Bakry-Qian \cite{BQ} for manifolds with densities. The generalization to arbitrary $p \in (1,\infty)$ has been proved by Valtorta \cite{Val} for $K=0 $ and Naber-Valtorta \cite{NaVal} for any $K\in \R$. 
All of these results hold for Riemannian manifolds. 

Regarding metric measure spaces,  the sharp Lichnerowitz spectral gap for $p=2$ was proved by Lott-Villani \cite{LV} under the $\CD(K,N)$ condition.  Jiang-Zhang \cite{JZ1} recently showed, still for $p=2$, that the improved version under an upper diameter bound holds for $\RCD^*(K,N)$ metric measure  spaces.   For Ricci limit spaces, in the case $K>0$ and   $D=\pi \sqrt{(N-1)/K}$, the $p$-spectral 
gap above has been recently obtained by Honda \cite{Honda} via proving  the stability of $\lambda^{1,p}$ 
under mGH convergence of compact Riemannian manifolds; this approach  was inspired by  the celebrated work of Cheeger-Colding \cite{CC3} where, in particular, it was shown the stability of $\lambda^{1,2}$ under mGH convergence. 
We also obtain the \emph{almost rigidity} for the $p$-spectral gap:  if an almost equality in the $p$-spectral gap holds, then the space must have almost maximal diameter.

\begin{theorem}[Theorem \ref{T:almostrigid}]\label{T:introalmost}
Let  $N> 1$, and $p \in (1,\infty)$ be fixed. Then for every $\ve>0$ there exists $\delta=\delta(\ve, N,p)$ such that the following holds. 

Let $(X,\sfd,\mm)$ be   an essentially non-branching metric measure space satisfying $\CD^{*}(N-1-\delta,N+\delta)$. If  $\lambda^{1,p}_{(X,\sfd,\mm)}\leq \lambda^{1,p}_{N-1,N,\pi} + \delta$, then  ${\diam}(X)\geq \pi -\ve$.
\end{theorem}

As a consequence, by a compactness argument and using the Maximal Diameter Theorem proved recently for $\RCD^*(K,N)$ by Ketterer \cite{Ket}, we have the following $p$-Obata and almost $p$-Obata Theorems.

\begin{corollary}[$p$-Obata Theorem]\label{Corol:p-Obata}
Let $(X,\sfd,\mm)$ be   an $\RCD^* (N-1,N)$ space for some $N \geq 2,$ and let $1<p<\infty$. If  
$$\lambda^{1,p}_{(X,\sfd,\mm)}= \lambda^{1,p}_{N-1,N,\pi}(=\lambda^{1,p}(S^N)),$$
then $(X,\sfd,\mm)$ is a spherical suspension, i.e. there exists an $\RCD^*(N-2,N-1)$ space $(Y, \sfd_Y, \mm_Y)$  such that $(X,\sfd,\mm)$ is isomorphic to  $[0,\pi] \times_{\sin}^{N-1}  Y$.
\end{corollary}

\begin{corollary}[Almost $p$-Obata Theorem]\label{cor:Almostp-Obata}
Let $N\geq 2$, and $p \in (1,\infty)$ be fixed. Then for every $\ve>0$ there exists $\delta=\delta(\ve, N,p)>0$ such that the following holds. 

Let $(X,\sfd,\mm)$ be   an $\RCD^* (N-1-\delta,N+\delta)$ space. If  
$$\lambda^{1,p}_{(X,\sfd,\mm)}\leq \lambda^{1,p}_{N-1,N,\pi} + \delta,$$
then there exists an $\RCD^*(N-2,N-1)$ space $(Y, \sfd_Y, \mm_Y)$  such that 
$$\sfd_{mGH}\left( (X,\sfd,\mm),  [0,\pi] \times_{\sin}^{N-1}  Y \right) \leq \ve.  $$
\end{corollary}

Let us mention that the classical  Obata's Theorem for $\RCD^*(K,N)$-spaces,  i.e. the version of  Corollary \ref{Corol:p-Obata} for $p=2$,  was recently obtained by Ketterer \cite{Ket2} (see also \cite{JZ1}) with different methods.

Finally we  recall that the case $p = 1$ can be attacked using the identity $h_{(X,\sfd,\mm)} = \lambda^{1,1}_{(X,\sfd,\mm)}$, where 
$h_{(X,\sfd,\mm)}$ is the so-called Cheeger isoperimetric constant, see Section \ref{SS:hKND}. 
Therefore Theorem \ref{T:introSpectral}, Theorem \ref{T:introalmost}, Corollary  \ref{Corol:p-Obata} and  Corollary \ref{cor:Almostp-Obata}
for the case  $p=1$ follow from the analogous results proved for the isoperimetric profile in \cite{CM1}. Nevertheless for reader's convenience,  the case $p=1$ will be discussed in detail  in Section \ref{S:p=1}.

\subsection{Log-Sobolev and Talagrand inequality}

Given a m.m.s. $(X,\sfd,\mm)$, we say that it supports the Log-Sobolev inequality with 
constant $\alpha>0$ if for any Lipschitz function $f:X\to [0,\infty)$ with $\int_X f(x)\,  \mm(dx)=1$ it holds
\begin{equation}\label{E:intrologsobolev}
2\alpha \int_X f \log f \, \mm \leq \int_{\{f>0\}} \frac{|\nabla f|^2}{f} \mm.
\end{equation}
The largest constant $\alpha$, such that \eqref{E:intrologsobolev} holds for any  Lipschitz function $f:X\to [0,\infty)$ with $\int_X f(x)\,  \mm(dx)=1$, 
will be called Log-Sobolev constant of $(X,\sfd,\mm)$ and denoted with $\alpha^{LS}_{(X,\sfd,\mm)}$.

Log-Sobolev inequality is already known  \cite[Theorem 30.22]{Vil} for essentially non-branching metric measure spaces satisfying $\CD(K,\infty)$ with $K >0$ with sharp constant $\alpha=K$, but it is an open problem (see for instance \cite[Open Problem 21.6]{Vil}) to get the sharp dimensional  constant $\alpha_{K,N}=\frac{KN}{N-1}$ for metric measure spaces with $N$-Ricci curvature bounded below by $K$. This is the goal of the next result.

As already done above, let us introduce the model constant for the one-dimensional case. 
Given $K\in \R$, $N\geq1$, $D\in (0,+\infty)$ we denote with $\alpha^{LS}_{K,N,D}>0$ the maximal constant $\alpha$ such that 
\begin{equation}\label{eq:intrologsobolevKND}
2\,\alpha \int_\R f \log f \, \mu \leq \int_{\{f>0\}} \frac{|f'|^2}{f} \mu, \quad \forall \mu\in {\mathcal F}_{K,N,D},
\end{equation}
for every Lipschitz $f:\R\to [0,\infty)$  with  $\int f \, \mu=1$.
  
\begin{remark}\label{rem:besta}
If $K>0$ and $D=\pi \sqrt{\frac{N-1}{K} }$, it is known that the corresponding optimal 
Log-Sobolev constant is $\frac{KN}{N-1}$ (for more details see the discussion in Section \ref{SS:Log-Sob}). 
It is an interesting open problem, that we don't address here, 
to give an explicit expression of the quantity  $\alpha^{LS}_{K,N,D}$ for general $K\in \R, N\geq 1$, $D \in (0, \infty)$. 
\end{remark}


\begin{theorem}[Sharp Log-Sobolev inequality, Theorem \ref{thm:compLS}]\label{T:introcompLS}
Let $(X,\sfd,\mm)$ be a metric measure space with  diameter $D \in (0,\infty)$ and satisfying $\CD^{*}(K,N)$ for some $K\in \R, N \in (1,\infty)$.  
Assume moreover it is essentially non-branching. 

Then  for any Lipschitz function $f:X\to [0,\infty)$ with $\int_X f\,  \mm=1$ it holds
$$
2\,\alpha^{LS}_{K,N,D} \int_X f \log f \, \mm \leq \int_{\{f>0\}} \frac{|\nabla f|^2}{f} \mm.
$$
In other terms it holds $\alpha^{LS}_{(X,\sfd,\mm)}\geq \alpha^{LS}_{K,N,D}$.
\\As a consequence, if $K>0$ and no diameter upper bound is assumed or   $D=\pi \sqrt{\frac{N-1}{K} }$, then $\alpha^{LS}_{K,N}=\frac{KN}{N-1}$ i.e.  for any Lipschitz function $f:X\to [0,\infty)$ with $\int_X f\,  \mm=1$ it holds
$$
\frac{2 KN}{N-1}  \int_X f \log f \, \mm \leq \int_{\{f>0\}} \frac{|\nabla f|^2}{f} \mm.
$$
\end{theorem}

In order to state the Talagrand inequality let us recall that the relative entropy functional $Ent_\mm: {\mathcal P}(X)\to [0, +\infty]$ with respect to a given $\mm\in {\mathcal P}(X)$ is defined to be 
$$
Ent_\mm(\mu)= \int_X \varrho\, \log \varrho \, \mm, \quad \text{ if } \mu=\varrho \mm 
 $$
and $+\infty$ otherwise.
Otto-Villani \cite{OtVil} proved that for smooth Riemannian manifolds the Log-Sobolev inequality with constant $\alpha>0$ 
implies the Talagrand inequality with constant $\frac{2}{\alpha}$ preserving sharpness.  The result was then generalized to arbitrary metric measure spaces by Gigli-Ledoux \cite{GL}.

Combining this result with  Theorem \ref{T:introcompLS} we get the following corollary which improves the 
Talagrand constant $2/K$, which is sharp for $\CD(K,\infty)$ spaces, by a factor $N-1/N$ in case the dimension is bounded above by $N$. 
This constant is sharp for $\CD^*(K,N)$ (or $\CD_{loc}(K,N)$) spaces, indeed it is sharp already in the smooth setting \cite[Remark 22.43]{Vil}.  
Since both our proof of the sharp Log-Sobolev inequality and the proof of Theorem \ref{thm:LST}  
are essentially optimal transport based, the following  can be seen as an answer to \cite[Open Problem 22.44]{Vil}.

\begin{theorem}[Sharp Talagrand inequality]\label{thm:talagrand}
Let $(X,\sfd,\mm)$ be a metric measure space with diameter $D \in (0,\infty)$, satisfying $\CD^{*}(K,N)$ for some $K\in \R, N\in (1,\infty)$,  
and assume moreover it is essentially non-branching and $\mm(X)=1$. 

Then it supports the Talagrand inequality with constant $\frac{2}{\alpha^{LS}_{K,N,D}}$, 
where  $\alpha^{LS}_{K,N,D}$ was defined in \eqref{eq:intrologsobolevKND}, i.e. it holds
$$
W_2^2(\mu, \mm) \leq \frac{2}{\alpha^{LS}_{K,N,D}} Ent_\mm(\mu)  \quad \text{for all }\mu \in \mathcal{P}(X). 
$$
In particular, if $K>0$  and no upper bound on the diameter is assumed or $D=\pi \sqrt{\frac{N-1}{K} }$, then  
$$
W_2^2(\mu, \mm) \leq \frac{2(N-1)}{KN} Ent_\mm(\mu) \quad \text{for all }\mu \in \mathcal{P}(X),
$$
the constant in the last inequality being sharp.
\end{theorem}

\subsection{Sobolev inequality}
Sobolev inequalities have been studied in many different contexts and many papers and books are devoted to this family of inequalities. 
Here we only mention two references mainly dealing with them in the Riemannian manifold case and the smooth $\CD$ condition case, 
respectively \cite{Hebey} and \cite{Led-Toul}.

We say that $(X,\sfd,\mm)$ supports a $(p,q)$-Sobolev inequality with constant $\alpha^{p,q}$ if for any $f : X \to \R$ Lipschitz function it holds
\begin{equation}\label{E:introSobolev}
\frac{\alpha^{p,q}}{p-q} \left\{ \left( \int_{X}  |f |^{p} \, \mm \right)^{\frac{q}{p}} -   \int_{X}  |f |^{q} \, \mm \right\} \leq   \int_{X}  |\nabla f |^{q} \, \mm, 
\end{equation}
and the largest constant $\alpha^{p,q}$ such that \eqref{E:introSobolev} holds for any Lipschitz function $f$ will be called the $(p,q)$-Sobolev constant 
of $(X,\sfd, \mm)$and will be denoted by $\alpha^{p,q}_{(X,\sfd,\mm)}$. 

A Sobolev inequality is known to hold for essentially non-branching m.m.s. satisfying $\CD(K,N)$, provided $K<0$, see \cite[Theorem 30.23]{Vil} and other Sobolev-type inequalities have been obtained in \cite{LV} for $\CD(K,N)$ spaces. Let us also mention \cite{Profeta} where the sharp $(2^{*},2)$-Sobolev inequality has been established for $\RCD^{*}(K,N)$-spaces, $K>0$, $N\in(2,\infty)$.
The goal here is to give a Sobolev inequality with sharp constant for essentially non-branching $\CD^*(K,N)$ spaces, $K\in \R, N> 1$,   taking  also into account an upper diameter bound.

\begin{theorem}[Sharp Sobolev inequality, Theorem \ref{thm:compL}]\label{T:introcompL}
Let $(X,\sfd,\mm)$ be a metric measure space with  diameter $D \in (0,\infty)$ and satisfying $\CD^{*}(K,N)$ for some $K\in \R, N\in (1,\infty)$.  
Assume moreover it is essentially non-branching. 

Then  for any Lipschitz function it holds
$$
\frac{\alpha^{p,q}_{K,N,D}}{p-q} \left\{ \left( \int_{X}  |f (x)|^{p} \, \mm(dx) \right)^{\frac{q}{p}} - \int_{X}  |f (x)|^{q} \, \mm(dx)   \right\} 
\leq    \int_{X}  |\nabla f (x)|^{q} \, \mm(dx), 
$$
where $\alpha^{p,q}_{K,N,D}$ is defined as the supremum among $\alpha > 0$ such that 
$$
\frac{\alpha}{p-q} \left\{ \left( \int_{X}  |f |^{p} \, \mu \right)^{\frac{q}{p}} -   \int_{X}  |f |^{q} \, \mu \right\} \leq   \int_{X}  |\nabla f |^{q} \, \mu, \ \forall \ f \in \Lip, \ \forall \ \mu \in \mathcal{F}_{K,N,D} . 
$$
In particular,  if $K>0, N>2$  and no upper bound on the diameter is assumed or $D=\pi \sqrt{\frac{N-1}{K} }$,  then  for any Lipschitz function $f$ it holds
$$
\frac{KN }{(p-2)(N-1)}  \left\{ \left( \int_{X}  |f |^{p} \, \mm \right)^{\frac{2}{p}} -   \int_{X}  |f |^{2} \, \mm \right\} \leq   \int_{X}  |\nabla f |^{2} \, \mm, 
$$
for any $2<  p \leq 2N/(N -2)$;  in other terms it holds $\alpha^{p,2}_{(X,\sfd,\mm)}\geq \frac{K N}{N-1}$.
\end{theorem}

This last result can be seen as  a solution to  \cite[Open Problem 21.11]{Vil}.

\section*{Acknowledgements}
The authors  wish to thank the Hausdorff center of Mathematics of Bonn, where part of the work has been developed,  for the excellent working conditions and the stimulating atmosphere during the trimester program ``Optimal Transport''.  A.M. is partly supported by the Swiss National Science Foundation.

\section{Prerequisites}\label{S:pre}

In what follows we say that a triple $(X,\sfd, \mm)$ is a metric measure space, m.m.s. for short, 
if $(X, \sfd)$ is a complete and separable metric space and $\mm$ is positive Radon measure over $X$. 
For this note we will only be concerned with m.m.s. with $\mm$ probability measure, that is $\mm(X) =1$, or at most with $\mm(X)<\infty$ which will be reduced to the probability case by a constant rescaling. 
The space of all Borel probability measure over $X$ will be denoted with $\mathcal{P}(X)$.

A metric space is a geodesic space if and only if for each $x,y \in X$ 
there exists $\gamma \in \Geo(X)$ so that $\gamma_{0} =x, \gamma_{1} = y$, with
$$
\Geo(X) : = \{ \gamma \in C([0,1], X):  \sfd(\gamma_{s},\gamma_{t}) = (s-t) \sfd(\gamma_{0},\gamma_{1}), s,t \in [0,1] \}.
$$
Recall that for complete geodesic spaces local compactness is equivalent to properness (a metric space is proper if every closed ball is compact).
We directly assume the ambient space $(X,\sfd)$ to be proper. Hence from now on we assume the following:
the ambient metric space $(X, \sfd)$ is geodesic, complete, separable and proper and $\mm(X) = 1$.

\medskip

We denote with $\mathcal{P}_{2}(X)$ the space of probability measures with finite second moment  endowed with the $L^{2}$-Wasserstein distance  $W_{2}$ defined as follows:  for $\mu_0,\mu_1 \in \mathcal{P}_{2}(X)$ we set
\begin{equation}\label{eq:Wdef}
  W_2^2(\mu_0,\mu_1) = \inf_{ \pi} \int_X \sfd^2(x,y) \, \pi(dxdy),
\end{equation}
where the infimum is taken over all $\pi \in \mathcal{P}(X \times X)$ with $\mu_0$ and $\mu_1$ as the first and the second marginal.
Assuming the space $(X,\sfd)$ to be geodesic, also the space $(\mathcal{P}_2(X), W_2)$ is geodesic. 

Any geodesic $(\mu_t)_{t \in [0,1]}$ in $(\mathcal{P}_2(X), W_2)$  can be lifted to a measure $\nu \in {\mathcal {P}}(\Geo(X))$, 
so that $({\rm e}_t) \, \sharp \, \nu = \mu_t$ for all $t \in [0,1]$. 
Here for any $t\in [0,1]$,  ${\rm e}_{t}$ denotes the evaluation map: 
$$
  {\rm e}_{t} : \Geo(X) \to X, \qquad {\rm e}_{t}(\gamma) : = \gamma_{t}.
$$

Given $\mu_{0},\mu_{1} \in \mathcal{P}_{2}(X)$, we denote by 
$\Opt(\mu_{0},\mu_{1})$ the space of all $\nu \in \mathcal{P}(\Geo(X))$ for which $({\rm e}_0,{\rm e}_1) \, \sharp\, \nu$ 
realizes the minimum in \eqref{eq:Wdef}. If $(X,\sfd)$ is geodesic, then the set  $\Opt(\mu_{0},\mu_{1})$ is non-empty for any $\mu_0,\mu_1\in \mathcal{P}_2(X)$.
It is worth also introducing the subspace of $\mathcal{P}_{2}(X)$
formed by all those measures absolutely continuous with respect to $\mm$: it is denoted by $\mathcal{P}_{2}(X,\sfd,\mm)$.


\subsection{Geometry of metric measure spaces}\label{Ss:geom}
Here we briefly recall the synthetic notions of lower Ricci curvature bounds, for more detail we refer to  \cite{BS10,lottvillani:metric,sturm:I, sturm:II, Vil}.

In order to formulate curvature properties for $(X,\sfd,\mm)$ we introduce the following distortion coefficients: given two numbers $K,N\in \erre$ with $N\geq1$, we set for $(t,\theta) \in[0,1] \times \erre_{+}$,
\begin{equation}\label{E:sigma}
\sigma_{K,N}^{(t)}(\theta):= 
\begin{cases}
\infty, & \textrm{if}\ K\theta^{2} \geq N\pi^{2}, \crcr
\displaystyle  \frac{\sin(t\theta\sqrt{K/N})}{\sin(\theta\sqrt{K/N})} & \textrm{if}\ 0< K\theta^{2} <  N\pi^{2}, \crcr
t & \textrm{if}\ K \theta^{2}<0 \ \textrm{and}\ N=0, \ \textrm{or  if}\ K \theta^{2}=0,  \crcr
\displaystyle   \frac{\sinh(t\theta\sqrt{-K/N})}{\sinh(\theta\sqrt{-K/N})} & \textrm{if}\ K\theta^{2} \leq 0 \ \textrm{and}\ N>0.
\end{cases}
\end{equation}

We also set, for $N\geq 1, K \in \R$ and $(t,\theta) \in[0,1] \times \erre_{+}$
\begin{equation} \label{E:tau}
\tau_{K,N}^{(t)}(\theta): = t^{1/N} \sigma_{K,N-1}^{(t)}(\theta)^{(N-1)/N}.
\end{equation}

As we will consider only the case of essentially non-branching spaces, we recall the following definition. 
\begin{definition}\label{D:essnonbranch}
A metric measure space $(X,\sfd, \mm)$ is \emph{essentially non-branching} if and only if for any $\mu_{0},\mu_{1} \in \mathcal{P}_{2}(X)$
which are absolutely continuous with respect to $\mm$ any element of $\Opt(\mu_{0},\mu_{1})$ is concentrated on a set of non-branching geodesics.
\end{definition}

A set $F \subset \Geo(X)$ is a set of non-branching geodesics if and only if for any $\gamma^{1},\gamma^{2} \in F$, it holds:
$$
\exists \ \bar t \in (0, 1) \, \colon \ \gamma_{ t}^{1} = \gamma_{t}^{2}, \ \forall \ t \in (0, \bar t) 
\quad 
\Longrightarrow 
\quad 
\gamma^{1}_{s} = \gamma^{2}_{s}, \quad \forall s \in [0,1].
$$

\begin{definition}[$\CD$ condition]\label{D:CD}
An essentially non-branching m.m.s. $(X,\sfd,\mm)$ verifies $\mathsf{CD}(K,N)$  if and only if for each pair 
$\mu_{0}, \mu_{1} \in \mathcal{P}_{2}(X,\sfd,\mm)$ there exists $\nu \in \Opt(\mu_{0},\mu_{1})$ such that
\begin{equation}\label{E:CD}
\r_{t}^{-1/N} (\gamma_{t}) \geq  \tau_{K,N}^{(1-t)}(\sfd( \gamma_{0}, \gamma_{1}))\r_{0}^{-1/N}(\gamma_{0}) 
 + \tau_{K,N}^{(t)}(\sfd(\gamma_{0},\gamma_{1}))\r_{1}^{-1/N}(\gamma_{1}), \qquad \nu\text{-a.e. } \, \gamma \in \Geo(X),
\end{equation}
for all $t \in [0,1]$, where ${\rm e}_{t} \, \sharp \, \nu = \r_{t} \mm$.
\end{definition}

For the general definition of $\CD(K,N)$ see \cite{lottvillani:metric, sturm:I, sturm:II}. It is worth recalling that if $(M,g)$ is a Riemannian manifold of dimension $n$ and 
$h \in C^{2}(M)$ with $h > 0$, then the m.m.s. $(M,g,h \, vol)$ verifies $\CD(K,N)$ with $N\geq n$ if and only if  (see Theorem 1.7 of \cite{sturm:II})
$$
Ric_{g,h,N} \geq  K g, \qquad Ric_{g,h,N} : =  Ric_{g} - (N-n) \frac{\nabla_{g}^{2} h^{\frac{1}{N-n}}}{h^{\frac{1}{N-n}}}.  
$$
In particular if $N = n$ the generalized Ricci tensor $Ric_{g,h,N}= Ric_{g}$ makes sense only if $h$ is constant. In particular, if $I \subset \R$ is any interval, $h \in C^{2}(I)$ 
and $\mathcal{L}^{1}$ is the one-dimensional Lebesgue measure, the m.m.s. $(I ,|\cdot|, h \mathcal{L}^{1})$ verifies $\CD(K,N)$ if and only if  
\begin{equation}\label{E:CD-N-1}
\left(h^{\frac{1}{N-1}}\right)'' + \frac{K}{N-1}h^{\frac{1}{N-1}} \leq 0,
\end{equation}
and verifies $\CD(K,1)$ if and only if $h$ is constant.

We also mention the more recent Riemannian curvature dimension condition $\RCD^{*}$ introduced in the infinite dimensional case in \cite{AGS11b,AGS,AGMR12} 
and in the finite dimensional case in \cite{EKS,AMS}. We refer to these papers and references therein for a general account 
on the synthetic formulation of Ricci curvature lower bounds for metric measure spaces. 
Here we only mention that $\RCD^{*}(K,N)$ condition 
is an enforcement of the so called reduced curvature dimension condition, denoted by $\CD^{*}(K,N)$, that has been introduced in \cite{BS10}: 
in particular the additional condition is that the Sobolev space $W^{1,2}(X,\mm)$ is an Hilbert space, see \cite{AGS11a, AGS11b}.

The reduced $\CD^{*}(K,N)$ condition asks for the same inequality \eqref{E:CD} of $\CD(K,N)$ but  the
coefficients $\tau_{K,N}^{(t)}(\sfd(\gamma_{0},\gamma_{1}))$ and $\tau_{K,N}^{(1-t)}(\sfd(\gamma_{0},\gamma_{1}))$ 
are replaced by $\sigma_{K,N}^{(t)}(\sfd(\gamma_{0},\gamma_{1}))$ and $\sigma_{K,N}^{(1-t)}(\sfd(\gamma_{0},\gamma_{1}))$, respectively.

Hence while the distortion coefficients of the $\CD(K,N)$ condition 
are formally obtained imposing one direction with linear distortion and $N-1$ directions affected by curvature, 
the $\CD^{*}(K,N)$ condition imposes the same volume distortion in all the $N$ directions. 

It was proved in \cite{RS2014} that the $\RCD^*(K,N)$ condition implies the essentially non-branching property, so this is a fairly natural assumption 
in the framework of m.m.s. satisfying lower Ricci bounds.

For both $\CD$-$\CD^*$ definitions there is a local version that is of some relevance for our analysis. Here we state only the local formulation $\mathsf{CD}(K,N)$, the one for  $\mathsf{CD}^{*}(K,N)$
being similar.

\begin{definition}[$\CD_{loc}$ condition]\label{D:loc}
An essentially non-branching m.m.s. $(X,\sfd,\mm)$ satisfies $\CD_{loc}(K,N)$ if for any point $x \in X$ 
there exists a neighborhood $X(x)$ of $x$ such that for each pair 
$\mu_{0}, \mu_{1} \in \mathcal{P}_{2}(X,\sfd,\mm)$ supported in $X(x)$
there exists $\nu \in \Opt(\mu_{0},\mu_{1})$ such that \eqref{E:CD} holds true for all $t \in [0,1]$.
The support of ${\rm e}_{t} \, \sharp \, \nu$ is not necessarily contained in the neighborhood $X(x)$.
\end{definition}

One of the main properties of the reduced curvature dimension condition is the globalization one:  
under the essentially non-branching property,  $\mathsf{CD}^{*}_{loc}(K,N)$ and $\mathsf{CD}^{*}(K,N)$ are equivalent (see \cite[Corollary 5.4]{BS10}).  Let us mention that the local-to-global property is satisfied also by the $\RCD^*(K,N)$ condition, see \cite{AMSLocToGlob}.

We also recall few relations between $\CD$ and $\CD^{*}$.
It is known by \cite[Theorem 2.7]{GigliMap} that, if $(X,\sfd,\mm)$ is a non-branching metric measure space 
satisfying $\CD(K,N)$ and $\mu_{0}, \mu_{1} \in \mathcal{P}(X)$ with $\mu_{0}$ absolutely continuous with respect to $\mm$, 
then there exists a unique optimal map $T : X \to X$ such that $(id, T) \, \sharp\, \mu_{0}$ realizes the minimum in \eqref{eq:Wdef} and the set 
$\Opt(\mu_{0},\mu_{1})$ contains only one element. The same proof holds if one replaces the non-branching assumption with the more general 
one of essentially non-branching, see for instance \cite{RS2014}.

\begin{remark}[$\CD^*(K,N)$ Vs $\CD_{loc}(K,N)$]\label{rk:CDCDs}
Results of \cite{BS10} imply the following chain of implications: if $(X,\sfd, \mm)$ is a proper, essentially non-branching, metric measure space, 
then 
$$
\CD_{loc}(K,N) \iff \CD^{*}_{loc}(K,N) \iff \CD^{*}(K,N), 
$$
provided $K,N \in \R$ with $N > 1$ or $N=1$ and $K \geq 0$. Let us remark  that on the other hand $\CD^*(K,1)$ does not imply $\CD_{loc}(K,1)$ for $K<0$: indeed it is possible to check that $(X,\sfd,\mm)=([0,1], |\cdot|, c\, \sinh(\cdot) \L^1)$ satisfies $\CD^*(-1,1)$ but not $\CD_{loc}(-1,1)$ which would require the density to be constant.
Hence 
 $\CD^{*}(K,N)$ and $\CD_{loc}(K,N)$ are equivalent if  $1 < N <\infty$ or $N =1$ and $K \geq 0$, but for $N =1$ and $K < 0$ the $\CD_{loc}(K,N)$ condition is strictly stronger than $\CD^{*}(K,N)$.


Note also that many results presented in \cite{BS10} are for metric measure spaces verifying $\CD(K-,N)$ (and its local version), that is they verify the 
$\CD(K',N)$ condition for all $K' < K$. Thanks to uniqueness of geodesics in $(\mathcal{P}_{2}(X), W_{2})$ guaranteed by the essentially non-branching assumption,  $\CD(K-,N)$ is equivalent to $\CD(K,N)$.

As a final comment we also mention that, for $K > 0$, $\CD^{*}(K,N)$ implies $\CD(K^{*},N)$ where $K^{*} = K(N-1) / N$.
For a deeper analysis on the interplay between $\CD^{*}$ and $\CD$ we refer to \cite{cava:decomposition, cavasturm:MCP}.
\end{remark}

\subsection{Measured Gromov-Hausdorff convergence and stability of $\RCD^*(K,N)$}\label{SS:mGHConv}

Let us first recall the notion of measured Gromov-Hausdorff convergence, mGH for short. Since in this work we will apply it to compact  m.m. spaces endowed with  probability measures having full support, we will restrict to this framework for simplicity (for a more general treatment see for instance \cite{GMS2013}).
   
 \begin{definition}  
 A sequence $(X_j,\sfd_j,\mm_j)$ of compact m.m. spaces with $\mm_j(X_j)=1$ and $\supp(\mm_j)=X_j$ is said to converge 
in the  measured Gromov-Hausdorff topology (mGH for short) to a compact m.m. space 
$(X_\infty,\sfd_\infty,\mm_\infty)$ with $\mm_\infty(X)=1$ and $\supp(\mm_\infty)=X_\infty$ if and only if there 
exists a separable metric space $(Z,\sfd_Z)$ and isometric embeddings  
$\{\iota_j:(X,\sfd_j)\to (Z,\sfd_Z)\}_{i \in \bar{\N}}$ such that
for every 
$\varepsilon>0$  there exists $j_0$ such that for every $j>j_0$
\[
\iota_\infty(X_\infty) \subset B^Z_{\varepsilon}[\iota_j (X_j)]  \qquad \text{and} \qquad  \iota_j(X_j) \subset B^Z_{\varepsilon}[\iota_\infty(X_\infty)], 
\]
where $B^Z_\varepsilon[A]:=\{z \in Z: \, \sfd_Z(z,A)<\varepsilon\}$ for every subset $A \subset Z$, and 
\[
\int_Z \varphi \,  ((\iota_j)_\sharp(\mm_j))\qquad    \to \qquad  \int_Z \varphi \,  ((\iota_\infty)_\sharp(\mm_\infty)) \qquad \forall \varphi \in C_b(Z), 
\]
where $C_b(Z)$ denotes the set of real valued bounded continuous functions  in $Z$.
 \end{definition}
 
The following theorem summarizes the compactness/stability properties we will use  in the proof of the almost rigidity result (notice these hold more generally for every $K\in \R$ by replacing mGH with  \emph{pointed}-mGH convergence).
\begin{theorem}[Metrizability and Compactness]\label{thm:CompRCD}
Let $K>0, N>1$ be fixed.  Then the mGH convergence restricted to  (isomorphism classes of)  $\RCD^*(K,N)$ spaces is metrizable by a distance function $\sfd_{mGH}$. Furthermore every sequence $(X_j,\sfd_j, \mm_j)$ of $\RCD^*(K,N)$ spaces admits a subsequence which  mGH-converges  to a limit $\RCD^*(K,N)$ space.
\end{theorem}

The compactness follows by the standard argument of Gromov, indeed for fixed $K>0,N>1$, the spaces have uniformly bounded diameter,  moreover  the measures of $\RCD^*(K,N)$  spaces  are uniformly doubling, hence the spaces  are uniformly totally bounded and thus compact in the GH-topology; the weak compactness of the measures follows using the doubling condition again and the fact that they are normalized. For the stability of the $\RCD^*(K,N)$ condition under mGH convergence see for instance   \cite{BS10,EKS,GMS2013}. The metrizability of mGH convergence restricted to a class of  uniformly doubling  normalized m.m. spaces having uniform diameter bounds is also well known, see for instance \cite{GMS2013}.

\subsection{Warped product} 

Given two geodesic m.m.s. $(B,\sfd_{B}, \mm_{b})$ and $(F,\sfd_{F},\mm_{F})$ and a Lipschitz function $f : B \to \R_{+}$ one can define a 
length function on the product $B \times F$: for any absolutely continuous curve $\gamma : [0,1] \to B \times F$ with $\gamma = (\alpha, \beta)$, 
define 
$$
L(\gamma) : = \int_{0}^{1} \left(  |\dot \alpha|^{2}(t) + (f\circ \alpha)^{2}(t) |\dot \beta|^{2}(t) \right)^{1/2} dt
$$
and define accordingly the pseudo-distance 
$$
|(p,x),(q,y)| : = \inf \left\{ L(\gamma) \colon \gamma_{0} = (p,x), \ \gamma_{1} = (q,y) \right\}.
$$
Then the warped product of $B$ with $F$  is defined as 
$$
B \times_{f} F : = \left( \bigslant{B\times F}{ \sim}, |\cdot, \cdot | \right),
$$
where $(p,x) \sim (q,y)$ if and only if $|(p,x), (q,y)| = 0$. One can also associate a measure and obtain the following object
$$
B\times^{N}_{f} F : = (B \times_{f} F, \mm_{C}), \qquad \mm_{C} : = f^{N} \mm_{B} \otimes \mm_{F}. 
$$
Then  $B\times^{N}_{f} F$ will be a metric measure space called measured warped product. For a general picture on the curvature properties of warped products, 
we refer to \cite{Ket}.


\medskip

\subsection{Localization method}

\medskip
The next theorem  represents the key technical tool of the present paper. The roots of such a result, known in literature as localization technique,  can be traced back to a work of   Payne-Weinberger \cite{PW} further developed  in the Euclidean space  by Gromov-Milman \cite{GrMi}, Lov\'asz-Simonovits \cite{LoSi} and Kannan-Lov\'asz-Simonovits \cite{KaLoSi}. The basic idea  consists in reducing an $n$-dimensional problem to a one dimensional one  via tools of convex geometry. Recently Klartag \cite{klartag}  found an $L^1$-optimal transportation approach leading to a  generalization of these ideas to  Riemannian manifolds;   the authors  \cite{CM1}, via a careful analysis avoiding any smoothness assumption,  generalized this approach to metric measure spaces.

\begin{theorem}\label{T:localize}
Let $(X,\sfd, \mm)$ be an essentially non-branching metric measure space with $\mm(X)=1$ satisfying $\CD_{loc}(K,N)$ for some $K,N \in \R$ and $N \in [1, \infty)$. 
Let $f : X \to \R$ be $\mm$-integrable such that $\int_{X} f\, \mm = 0$ and assume the existence of $x_{0} \in X$ such that $\int_{X} | f(x) |\,  \sfd(x,x_{0})\, \mm(dx)< \infty$. 
\medskip

Then the space $X$ can be written as the disjoint union of two sets $Z$ and $\mathcal{T}$ with $\mathcal{T}$ admitting a partition 
$\{ X_{q} \}_{q \in Q}$, where each $X_{q}$ is the image of a geodesic; moreover there exists a family of probability measures $\{\mm_{q} \}_{q \in Q} \subset \mathcal{P}(X)$ with the following properties: 

\begin{itemize}
\item For any $\mm$-measurable set $B \subset \mathcal{T}$ it holds 
$$
\mm(B) = \int_{Q} \mm_{q}(B) \, \qq(dq), 
$$
where $\qq$ is a probability measure over $Q \subset X$.
\medskip
\item For $\qq$-almost every $q \in Q$, the set $X_{q}$ is a geodesic with strictly positive length and $\mm_{q}$ is supported on it. 
Moreover $q \mapsto \mm_{q}$ is a $\CD(K,N)$ disintegration, that is $\mm_{q} = g(q,\cdot) \, \sharp \, \left( h_{q} \cdot \mathcal{L}^{1} \right)$, with 
\begin{equation}\label{E:curvdensmm}
h_{q}( (1-s)  t_{0}  + s t_{1} )^{\frac{1}{N-1}}  
 \geq \sigma^{(1-s)}_{K,N-1}(t_{1} - t_{0}) h_{q} (t_{0})^{\frac{1}{N-1}} + \sigma^{(s)}_{K,N-1}(t_{1} - t_{0}) h_{q} (t_{1})^{\frac{1}{N-1}},
\end{equation}
for all $s\in [0,1]$ and for all $t_{0}, t_{1} \in \dom(g(q,\cdot))$ with  $t_{0} < t_{1}$, where $g(q,\cdot)$ is the isometry with range $X_{q}$.
If $N =1$, for $\qq$-a.e. $q \in Q$ the density $h_{q}$ is constant.
\medskip
\item For $\qq$-almost every $q \in Q$, it holds $\int_{X_{q}} f \, \mm_{q} = 0$ and $f = 0$ $\mm$-a.e. in $Z$.
\end{itemize}
\end{theorem} 
 
\medskip

\begin{remark}
Inequality \eqref{E:curvdensmm} is the weak formulation of the following differential inequality on $h_{q,t_{0},t_{1}}$: 
\begin{equation}\label{eq:hqt0t1}
\left(h_{q,t_{0},t_{1}}^{\frac{1}{N-1}}\right)'' + (t_{1}-t_{0})^{2} \frac{K}{N-1}h_{q,t_{0},t_{1}}^{\frac{1}{N-1}} \leq 0,
\end{equation}
for all $t_{0}<t_{1} \in \dom(g(q,\cdot))$, where $h_{q,t_{0},t_{1}} (s) : = h_{q} ((1-s)t_{0} + st_{1})$. It is easy to observe that the differential inequality  \eqref{eq:hqt0t1}  on $h_{q,t_{0},t_{1}}$ is equivalent to the following differential inequality on $h_q$:
$$
\left(h_{q}^{\frac{1}{N-1}}\right)'' + \frac{K}{N-1}h_{q}^{\frac{1}{N-1}} \leq 0,
$$
that is precisely \eqref{E:CD-N-1}. Then Theorem \ref{T:localize} can be alternatively stated as follows. \\ 
\emph{
If $(X,\sfd,\mm)$ is an essentially non-branching m.m.s. verifying $\CD_{loc}(K,N)$ 
and $\f : X \to \R$ is a 1-Lipschitz function, then the corresponding decomposition of the space 
in maximal rays $\{ X_{q}\}_{q\in Q}$ produces a disintegration $\{\mm_{q} \}_{q\in Q}$ of $\mm$ so that  for $\qq$-a.e. $q\in Q$, 
$$
\textrm{the m.m.s. }(  \dom(g(q,\cdot)), |\cdot|, h_{q} \mathcal{L}^{1}) \quad \textrm{verifies} \quad \CD(K,N).
$$
}
Accordingly, from now on we will say that the disintegration $q \mapsto \mm_{q}$ is a $\CD(K,N)$ disintegration.
\end{remark}

Few comments on Theorem \ref{T:localize} are in order.
From \eqref{E:curvdensmm} it follows that
\begin{equation}\label{E:regularityh}
\{ t \in \dom(g(q,\cdot))  \colon  h_{q}(t) > 0 \}  \textrm{ is convex and} \quad t \mapsto h_{q}(t) \textrm{ is locally Lipschitz continuous}.
\end{equation}
The measure $\qq$ is the quotient measure associated to the partition $\{ X_{q} \}_{q\in Q}$ of $\mathcal{T}$ and $Q$ its quotient set, see \cite{CM1} for details.

\bigskip
\bigskip

\section{Sharp Brunn-Minkowski inequality}

In this section we prove sharp Brunn-Minkowski inequality for m.m.s. satisfying $\CD_{loc}(K,N)$. 
It follows  from Remark \ref{rk:CDCDs} that the same result holds under $\CD^{*}(K,N)$ for any $K,N \in \R$, provided $N \in (1,\infty)$ or $N = 1$ and $K \geq 0$. See also Remark \ref{R:intro}.
The same will hold for all the inequalities proved in the paper.

\begin{theorem}\label{T:BM}
Let $(X,\sfd,\mm)$ with $\mm(X)<\infty$ verify $\CD_{loc}(K,N)$ for some $N,K \in \R$ and $N \in [1, \infty)$. Assume moreover $(X,\sfd,\mm)$ to be essentially non-branching.
Then it satisfies the following sharp Brunn-Minkowski inequality: for any $A_{0}, A_{1} \subset X$
\begin{equation}\label{E:BM}
\mm(A_{t})^{1/N}  \geq \tau_{K,N}^{(1-t)}(\theta) \mm(A_{0})^{1/N} + \tau_{K,N}^{(t)}(\theta) \mm(A_{1})^{1/N},
\end{equation}
where $A_{t}$ is the set of $t$-intermediate points between $A_{0}$ and $A_{1}$, that is 
$$
A_{t} = \ee_{t} \Big(  \{ \gamma \in \Geo(X) \colon \gamma_{0} \in A_{0}, \gamma_{1} \in A_{1}   \}  \Big),
$$
and $\theta$ the minimal/maximal length of geodesics from $A_{0}$ to $A_{1}$: 
$$
\theta:= 
\begin{cases}
\inf_{(x_{0},x_{1}) \in A_{0} \times A_{1}} \sfd(x_{0},x_{1}), & \textrm{if } K \geq 0, \\
\sup_{(x_{0},x_{1}) \in A_{0} \times A_{1}} \sfd(x_{0},x_{1}), & \textrm{if } K < 0.
\end{cases}
$$
\end{theorem}

Before starting the proof of Theorem \ref{T:BM} we recall the classical result of Borell \cite{Borell} and Brascamp-Lieb \cite{BraLi} 
characterizing one-dimensional measures satisfying Brunn-Minkowski inequality.  

\begin{lemma}\label{L:1dBM}
Let $\eta$ be a Borel measure defined on $\R$  admitting the following representation:  $\eta =  h \cdot \mathcal{L}^{1}$. 
The following are equivalent: 
\begin{itemize}
\item[$i)$] The density $h$ is $(K,N)$-concave on its convex support, that is 
$$
\left( h^{\frac{1}{N-1}} \right)'' + \frac{K}{N-1} h^{\frac{1}{N-1}} \leq 0, 
$$ 
in the weak sense, see \eqref{E:curvdensmm}.

\item[$ii)$] For any $A_{0}, A_{1}$ subsets of $\R$
$$
\eta(A_{t}) \geq \tau_{K,N}^{(1-t)}(\theta) \, \eta(A_{0})^{1/N} + \tau_{K,N}^{(t)}(\theta) \, \eta(A_{1})^{1/N},
$$
where $A_{t} : = \{ (1-t)x + t y : x \in A_{0}, \, y \in A_{1} \}$ and $\theta$ is the minimal/maximal length of geodesics from $A_{0}$ to $A_{1}$: 
$$
\theta:= 
\begin{cases}
{\rm ess}\inf_{(x_{0},x_{1}) \in A_{0} \times A_{1}} \sfd(x_{0},x_{1}), & \textrm{if } K \geq 0, \\
{\rm ess}\sup_{(x_{0},x_{1}) \in A_{0} \times A_{1}} \sfd(x_{0},x_{1}), & \textrm{if } K < 0.
\end{cases}
$$
\end{itemize}

\end{lemma}

For reader's convenience we include here a proof that $i)$ implies $ii)$, which is the implication we will use later. 

\begin{proof}
Consider the $N$-entropy: for any $\mu = \rho \cdot \eta$
$$
\mathcal{S}_{N} (\mu | \eta) : = - \int \rho^{-1/N}(x) \,\mu(dx).
$$
Observe that $ii)$ is implied by displacement convexity of $\mathcal{S}_{N}$ with respect to the $L^{2}$-Wasserstein distance over $(\R, |\cdot|)$. 
Just consider $\mu_{0} : =\eta(A_{0})^{-1} \eta \llcorner_{A_{0}}$ and $\mu_{1} : =\eta(A_{1})^{-1} \eta \llcorner_{A_{1}}$ and use Jensen's inequality.
Consider therefore a geodesic curve 
$$
[0,1] \ni t \mapsto \rho_{t} \, \eta \in W_{2}(\R,|\cdot |), \qquad T_{t} \, \sharp\, \rho_{0}\eta = \rho_{t}\,\eta,
$$
where $T_{t} = Id (1-t) + t T$ and $T$ is the ($\mu_0$-essentially) unique monotone rearrangement such that  $T \, \sharp \, \mu_{0} =\mu_{1}$. 
Thanks to approximate differentiability of $T$, one can use change of variable formula 
$$
\rho_{t}(T_{t}(x)) \, h(T_{t}(x)) \, |(1-t) + t T'|(x) = \rho_{0}(x) h(x)
$$
and obtain the following chain of equalities: 
\begin{align*}
\int_{\supp(\mu_{t})} \rho_{t}(x)^{\frac{N-1}{N}} \, \eta(dx) &~ =   \int_{\supp(\mu_{t})} \rho_{t}(x)^{\frac{N-1}{N}} h(x) \, dx \crcr
&~ = \int_{\supp(\mu_{0})} \rho_{t}(T_{t}(x))^{\frac{N-1}{N}} h(T_{t}(x)) |(1-t) + t T'| \, dx \crcr
&~ = \int_{\supp(\mu_{0})} \rho_{0}(x)^{\frac{N-1}{N}} \Big( \frac{h(T_{t}(x))}{h(x)} \Big)^{\frac{1}{N}} |(1-t) + t T'|(x)^{\frac{1}{N}} \, \eta(dx). 
\end{align*}
Hence the claim has became to prove that $t \mapsto J_{t}(x)^{\frac{1}{N}}$ is concave, where $J_{t}$ is the Jacobian of $T_{t}$ with respect to $\eta$ and 
$$
J_{t} (x) = J_{t}^{G}(x) \cdot J_{t}^{W}(x), \quad J_{t}^{G}(x) =|(1-t) + t T'|(x),\quad J_{t}^{W}(x) = \frac{h(T_{t}(x))}{h(x)},
$$
where $J^{G}$ is the geometric Jacobian and $J^{W}$ the weighted Jacobian. 
Since $t\mapsto J_{t}^{G}(x)$ is linear, using H\"older's inequality the claim follows straightforwardly from the $(K,N)$-convexity of $h$.
\end{proof}

We can now move to the proof of Theorem \ref{T:BM}. 

\begin{proof}[Proof of Theorem \ref{T:BM}]
First of all notice that up to replacing $\mm$ with the normalized measure $\frac{1}{\mm(X)} \, \mm$ we can assume that $\mm(X)=1$. Let $A_{0},A_{1} \subset X$ be two given Borel sets of positive $\mm$-measure. 
\medskip

{\bf Step 1.} 
Consider the function $f : = \chi_{A_{0}}/\mm(A_{0}) - \chi_{A_{1}}/\mm(A_{1})$ and observe that 
$\int_{X} f \, \mm = 0$. From Theorem \ref{T:localize}, the space  $X$ can be written as the disjoint union of two sets $Z$ and $\mathcal{T}$ 
with $\mathcal{T}$ admitting a partition  $\{ X_{q} \}_{q \in Q}$ and a corresponding disintegration 
of $\mm\llcorner_{\mathcal{T}}$, $\{\mm_{q} \}_{q \in Q}$ such that: 
$$
\mm\llcorner_{\mathcal{T}} = \int_{Q} \mm_{q} \, \qq(dq), 
$$
where $\qq$ is the quotient measure, for $\qq$-almost every $q \in Q$, the set $X_{q}$ is a geodesic, $\mm_{q}$ is supported on it 
and $q \mapsto \mm_{q}$ is a $\CD(K,N)$ disintegration.
Finally, for $\qq$-almost every $q \in Q$, it holds $\int_{X_{q}} f \, \mm_{q} = 0$ and $f = 0$ $\mm$-a.e. in $Z$.
We can also consider the trivial disintegration of $\mm$ restricted to $Z$ where each equivalence class is a single point: 
$$
\mm\llcorner_{Z} = \int_{Z} \delta_{z} \mm(dz),
$$
where $\delta_{z}$ stands for the Dirac delta in $z$. Then define $\hat \qq : = \qq + \mm\llcorner_{Z}$ and $\hat \mm_{q} = \mm_{q}$ if $q \in Q$ and 
$\hat \mm_{q} = \delta_{q}$ if $q \in Z$. Since $Q \cap Z = \emptyset$, the previous definitions are well posed and we have the following decomposition of 
$\mm$ on the whole space
$$
\mm = \int_{Q \cup Z} \hat \mm_{q} \, \hat \qq(dq).
$$

\medskip

{\bf Step 2.} Use the following notation $A_{0,q} : = A_{0} \cap X_{q}$, $A_{1,q} : = A_{1}\cap X_{q}$ and the set of $t$-intermediate points between 
$A_{0,q}$ and $A_{1,q}$ in $X_{q}$ is denoted with $A_{t, q}\subset X_{q}$. 
Then from Lemma \ref{L:1dBM}, for $\hat \qq$-a.e. $q \in Q$
$$
\mm_{q}(A_{t,q}) \geq \left( \tau_{K,N}^{(1-t)}(\theta) \mm_{q}(A_{0,q})^{1/N} + \tau_{K,N}^{(t)}(\theta) \mm_{q}(A_{1,q})^{1/N}\right)^{N}.  
$$
Since $\int f \mm_{q} = 0$ implies $\frac{\mm_q(A_{0,q})}{\mm(A_0)}=\frac{\mm_q(A_{1,q})}{\mm(A_1)}$, it follows that 
\begin{equation}\label{eq:mmqAtq}
\mm_{q}(A_{t,q}) \geq \frac{\mm_{q}(A_{0,q})}{\mm(A_{0})} \left( \tau_{K,N}^{(1-t)}(\theta) \mm(A_{0})^{1/N} + \tau_{K,N}^{(t)}(\theta) \mm(A_{1})^{1/N}\right)^{N}.  
\end{equation}
We now show that \eqref{eq:mmqAtq} holds also for $\hat \qq$-a.e. (or equivalently $\mm$-a.e.) $q\in Z$. Note that in this case $\mm_{q}$ has to be replaced by $\delta_{q}$. 
Since by construction $0=f= \chi_{A_{0}}/\mm(A_{0}) - \chi_{A_{1}}/\mm(A_{1})$ on $Z$, then  necessarily 
$$
\mm\left(Z\setminus\big((A_{0}\cap A_{1}) \cup (X\setminus \left( A_{0} \cup A_{1}\right)\big) \right)=0.
$$
It follows that if $Z$ does not have $\mm$-measure zero, we have two possibilities: 
$$
\mm \left(Z\cap \left(X\setminus \left( A_{0} \cup A_{1}\right)\right)\right)>0 , \quad \textrm{or} \quad  \mm(A_{0})=\mm(A_{1}) \text{ and } \mm\left( Z \cap \left(A_{0}\cap A_{1}\right)\right)>0.
$$
Therefore,  if $\mm(Z)>0$, for $\hat \qq$-a.e. (or equivalently $\mm$-a.e.) $q\in Z$  we have two possibilities: 
$$
q \in X\setminus \left( A_{0} \cup A_{1}\right), \quad \textrm{or} \quad q \in A_{0}\cap A_{1}.
$$
Interpreting the intermediate points as the point itself,  in the first case \eqref{eq:mmqAtq} (with $\mm_{q}$ replaced by $\delta_{q}$) holds trivially (i.e. we get $0 \geq 0$). In the second case  it reduces 
to  show that
$$
\left( \tau_{K,N}^{(1-t)}(\theta)  + \tau_{K,N}^{(t)}(\theta) \right)^{N} \leq 1.
$$
For $K \geq 0$, since we are in the case $\mm(A_{0} \cap A_{1})>0$, it follows that $\theta = 0$ and therefore  $\tau_{K,N}^{(t)}(\theta) = t$, 
proving the previous inequality. For $K < 0$, recalling that $K \to \sigma_{K,N}^{(t)}(\theta)$ is non-decreasing (see \cite{BS10}, Remark 2.2 ), 
by H\"older's inequality
$$
\left( \tau_{K,N}^{(1-t)}(\theta)  + \tau_{K,N}^{(t)}(\theta) \right)^{N} \leq ( 1-t + t)
\cdot \Big( \sigma_{K,N-1}^{(1-t)}(\theta) +  \sigma_{K,N-1}^{(t)}(\theta)\Big)^{N-1} \leq 1, 
$$
as desired. We have therefore proved that  
\begin{equation}\label{eq:completemmqAtq}
\hat \mm_{q}(A_{t,q}) \geq \frac{\hat \mm_{q}(A_{0,q})}{\mm(A_{0})} 
\left( \tau_{K,N}^{(1-t)}(\theta) \mm(A_{0})^{1/N} + \tau_{K,N}^{(t)}(\theta) \mm(A_{1})^{1/N}\right)^{N},
\end{equation}
for $\hat \qq$-a.e. $q \in Q \cup Z$. 
%
%
Taking the integral of \eqref{eq:completemmqAtq} in $q \in Q\cup Z$  one obtains that 
\begin{align*}
\mm(A_{t})  =  &~ \int_{Q\cup Z} \hat \mm_{q}(A_{t} \cap X_{q}) \, \hat  \qq(dq)  \crcr
\geq  &~ \int_{Q\cup Z} \hat \mm_{q}(A_{t,q}) \, \hat  \qq(dq)  \crcr
\geq &~ \left( \tau_{K,N}^{(1-t)}(\theta) \mm(A_{0})^{1/N} + \tau_{K,N}^{(t)}(\theta) \mm(A_{1})^{1/N}\right)^{N}  \int_{Q\cup Z} 
				\frac{\hat \mm_{q}(A_{0,q})}{\mm(A_{0})} \,\hat \qq(dq) \crcr
= &~ \left( \tau_{K,N}^{(1-t)}(\theta) \mm(A_{0})^{1/N} + \tau_{K,N}^{(t)}(\theta) \mm(A_{1})^{1/N}\right)^{N},
\end{align*}
and the claim follows.
\end{proof}

\section{$p$-Spectral gap}\label{S:Isop}
Given a metric space $(X,\sfd)$, we denote with $\Lip(X)$ (respectively $\Lip_c(X)$) the vector space of real valued Lipschitz functions (resp. with compact support). For a  Lipschitz function $f:X\to \R$ the local Lipschitz constant $|\nabla f|$ is defined by
$$|\nabla f|(x)=\limsup_{y\to x} \frac{|f(x)-f(y)|}{\sfd(x,y)} \quad  \text{if $x$ is not isolated, $0$ otherwise}. $$

For a m.m.s. $(X,\sfd,\mm)$, for every $p \in (1,\infty)$ we define \emph{the first eigenvalue $\lambda_{1,p}(X,\sfd,\mm)$ of the $p$-Laplacian} by
\begin{equation}\label{eq:deflamp}
\lambda^{1,p}_{(X,\sfd,\mm)}:=\inf \left\{\frac{\int_X |\nabla f|^p \, \mm}{\int_{X} |f|^p \, \mm} \,:\, f\in \Lip(X)\cap L^{p}(X,\mm),\,  f\neq 0,\,  \int_X f |f|^{p-2} \, \mm=0 \right\}.
\end{equation}

\subsection{$p$-spectral gap  for m.m.s. over $(\R, |\cdot|)$: the model spaces}\label{Ss:modelspectral}

Consider the following family of probability measures
\begin{eqnarray}
\mathcal{F}^{s}_{K,N,D} : = \{ \mu \in \mathcal{P}(\R) : &\supp(\mu) \subset [0,D], \, \mu = h_{\mu} \mathcal{L}^{1},\,
h_{\mu}\, \textrm{verifies} \, \eqref{E:curvdensmm} \ \textrm{and is continuous if } N\in (1,\infty), \nonumber \\ 
& \quad h_{\mu}\equiv \textrm{const} \text{ if }N=1   \},
\end{eqnarray}
where $D \in (0,\infty)$ and  the corresponding \emph{synthetic}  first  non-negative eigenvalue of the $p$-Laplacian 
$$
^s\lambda_{K,N,D}^{1,p}: = 
\inf_{\mu \in \mathcal{F}^s_{K,N,D}} \inf \left\{ \frac{\int_{\R} |u'|^{p}\, \mu }{\int_{\R} |u|^{p}\, \mu } 
: u \in \Lip(\R) \cap L^{p}(\mu), \, \int_{\R} u |u|^{p-2}\mu =0,\, u\neq 0  \right\}.
$$
The term synthetic refers to $\mu \in \mathcal{F}^{s}_{K,N,D}$ meaning that the Ricci curvature bound is satisfied in its synthetic formulation:
if $\mu = h \cdot \mathcal{L}^{1}$, then $h$ verifies \eqref{E:curvdensmm}.

\medskip

The first goal of this section is to  prove that $^s\lambda_{K,N,D}^{1,p}$ coincides with  its smooth counterpart $\lambda_{K,N,D}^{1,p}$ defined by
\begin{equation}\label{defLa1p}
\lambda_{K,N,D}^{1,p}: =  \inf_{\mu \in \mathcal{F}_{K,N,D}} \inf \left\{ \frac{\int_{\R} |u'|^{p}\, \mu }{\int_{\R} |u|^{p}\, \mu } : u \in \Lip(\R) \cap L^{p}(\mu), \, \int_{\R} u |u|^{p-2}\mu =0,\, u\neq 0  \right\},
\end{equation}
where now $\mathcal{F}_{K,N,D}$ denotes the set of $\mu \in \mathcal{P}(\R)$ such that  $\supp(\mu) \subset [0,D]$   
and $\mu = h \cdot \mathcal{L}^{1}$ with $h \in C^{2}((0,D))$ satisfying
\begin{equation}\label{eq:DiffIne}
\left( h^{\frac{1}{N-1}} \right)'' + \frac{K}{N-1} h^{\frac{1}{N-1}} \leq 0.
\end{equation}
It is easily verified that $\mathcal{F}_{K,N,D} \subset \mathcal{F}^{s}_{K,N,D}$.  

In order  to prove that $^s\lambda_{K,N,D}^{1,p}= \lambda_{K,N,D}^{1,p}$ the following approximation result, proved in \cite[Lemma 6.2]{CM1} will play a key role. 
In order to state it let us recall that a standard mollifier in $\R$ is a non negative $C^\infty(\R)$ 
function $\psi$ with compact support in $[0,1]$ such  that $\int_{\R} \psi = 1$.

\begin{lemma}\label{lem:approxh}
Let  $D \in (0,\infty)$ and let  $h:[0,D] \to [0,\infty)$ be a continuous function. Fix $N\in (1,\infty)$ and for $\ve>0$ define
\begin{equation}
h_{\ve}(t):=[h^{\frac{1}{N-1}}\ast \psi_{\ve} (t)]^{N-1}  := \left[ \int_{\R} h(t-s)^{\frac{1}{N-1}}  \psi_{\ve} (s) \, d s\right]^{N-1} 
										=  \left[ \int_{\R} h(s)^{\frac{1}{N-1}}  \psi_{\ve} (t-s) \, d s\right]^{N-1},
\end{equation}
where $\psi_\ve(x)=\frac{1}{\ve} \psi(x/\ve)$ and $\psi$ is a standard mollifier function. The following properties hold:
\begin{enumerate}
	\item $h_{\ve}$ is a non-negative $C^\infty$ function with support in $[-\ve, D+\ve]$; \medskip
	\item $h_{\ve}\to h$ uniformly  as $\ve \downarrow 0$, in particular $h_{\ve} \to h$ in $L^{1}$. \medskip
	\item If $h$ satisfies the convexity condition \eqref{E:curvdensmm} corresponding to the above fixed $N>1$ 
		and some $K \in \R$ then also $h_{\ve}$ does. In particular $h_{\ve}$ satisfies the differential inequality \eqref{eq:DiffIne}.
\end{enumerate}
\end{lemma}

%
%
%
%
%

\begin{proposition}\label{P:sLa=La}
For every $p \in (1,+\infty)$, $N\in [1,\infty)$, $K\in \R$, $D\in (0,\infty)$ it holds $^s\lambda^{1,p}_{K,N,D}=\lambda^{1,p}_{K,N,D}$.
\end{proposition}

\begin{proof}
First of all observe that  for $N=1$ clearly we have $\mathcal{F}_{K,N,D}=\mathcal{F}^{s}_{K,N,D}$ since the density $h_{\mu}$ has to be constant.  We can then  assume without loss of generality that $N\in(1,\infty)$.  
\\Since  $\mathcal{F}_{K,N,D} \subset \mathcal{F}^{s}_{K,N,D}$ then clearly $^s\lambda^{1,p}_{K,N,D}\leq \lambda^{1,p}_{K,N,D}$.
\\Assume by contradiction the inequality is strict. Then there exists a measure $\mu=h \cdot \mathcal{L}^{1} \in  \mathcal{F}^{s}_{K,N,D}$ and $\delta>0$ such that   
\begin{equation}\nonumber 
\lambda^{1,p}_{(\R,|\cdot|,\mu)}\leq \lambda^{1,p}_{K,N,D} - 2\delta.
\end{equation} 
Therefore, by  the very definition of $\lambda^{1,p}_{(\R,|\cdot|,\mu)}$, there exists $u \in \Lip(\R)$,  such that $u\neq 0$, $\int_{\R} u |u|^{p-2} \, h \, ds=0$ and
\begin{equation}\label{eq:contrad}
\int_\R |u'(s)|^p \,h(s) \, d s  \leq \Big(\lambda^{1,p}_{K,N,D} - \frac{3}{2} \delta\Big) \int_\R |u(s)|^p \, h(s) \, d s .
\end{equation} 
Now, Lemma \ref{lem:approxh} gives a sequence $h_{k}\in C^{\infty}(\R)$ such that 
\begin{equation}\label{eq:hku}
\supp (h_k)\subset \left[ -\frac{1}{k}, D+\frac{1}{k} \right], 
	\qquad   \mu_k:=h_k \cdot \mathcal{L}^{1} \in  \mathcal{F}_{K,N,D+\frac{2}{k}}, 
		\qquad  h_k \to h \; \text{uniformly on } [0,D]. 
\end{equation}
Called now $u_k:=u-c_k\in \Lip(\R)\cap L^p(\R,h_k\mathcal{L}^1)$ where $c_k\in \R$  are such that $\int_\R u_k |u_k|^{p-2} \, h_k \, ds=0$, thanks to \eqref{eq:hku} it holds $c_k\to 0$ and thus
$$\int_{\R} |u_k(s)|^p \, h_k(s)\, ds \to \int_{\R} |u(s)|^p \, h(s) \, ds \quad\text{and}\quad  \int_{\R} |u_k'(s)|^p \, h_k(s)\, ds \to \int_{\R} |u'(s)|^p \, h(s) \, ds. $$
Therefore \eqref{eq:contrad}, combined with the continuity of $\varepsilon \mapsto \lambda^{1,p}_{K,N,D+\varepsilon}$ (see Theorem \ref{thm:La1pComp} below), implies that for $k$ large enough one has
$$ \int_\R |u_k'(s)|^p \,h_k(s) \, d s  \leq (\lambda^{1,p}_{K,N,D} - \delta) \int_\R |u_k(s)|^p \, h_k(s) \, d s \leq  \big(\lambda^{1,p}_{K,N,D+\frac{2}{k}} - \frac{\delta}{2}\big) \int_\R |u_k(s)|^p \, h_k(s) \, d s, $$
contradicting the definition of $ \lambda^{1,p}_{K,N,D+\frac{2}{k}} $ given in \eqref{defLa1p}.
\end{proof}

The next goal of the section is to understand  the quantity $\lambda^{1,p}_{K,N,D}$.  Since now the density of the reference probability measure is smooth, we enter into a more classical framework where a number of people contributed.  The sharp $p$-spectral gap in case $K>0$ and without upper bounds on the diameter was obtained by Matei \cite{Matei}. The case $K=0$ and the diameter is bounded above was obtained in the sharp form by Valtorta \cite{Val}. Finally the  case $K<0$ and diameter bounded above was obtained in the sharp form by Naber-Valtorta \cite{NaVal}. Actually, as explained in their paper, the arguments   in \cite{NaVal} hold in the general case $K\in \R$, $N\in [1,\infty)$, provided one identifies the correct model space.   As usual, to describe the model space one has to examine separately the cases $K<0$, $K=0$ and $K>0$; in order to unify the presentation let us denote with $\tan_{K,N}(t)$ the following function:
\begin{equation}\label{eq:deftanK}
\tan_{K,N}(t):=
\begin{cases}
\sqrt{-K/(N-1)}\, \tanh(\sqrt{- K /(N-1)} t) & \quad \text{if } K<0, \crcr
0 & \quad \text{if } K=0 \crcr
\sqrt{K/(N-1)} \, \tan(\sqrt{K/(N-1)} t) & \quad \text{if } K>0.
\end{cases}
\end{equation}
Now, for each $K\in \R, N\in [1,\infty), D\in (0,\infty)$, let $\hat{\lambda}_{K,N,D}^{1,p}$ denote the first positive eigenvalue on $[-D/2,D/2]$ of the eigenvalue problem
\begin{equation}\label{eq:defbarla}
\frac{d}{dt} \left(\dot{w}^{(p-1)} \right) + (N-1)  \tan_{K,N} (t )\, \dot{w}^{(p-1)} + \hat{\lambda}_{K,N,D}^{1,p} w^{(p-1)} = 0.
\end{equation}
It is possible to show (see \cite{NaVal}) that $\hat{\lambda}_{K,N,D}^{1,p}$ is the unique value of $\hat{\lambda}$ such that the solution of 
$$
\begin{cases}
\dot{\phi}=\left(\frac{\hat{\lambda}}{p-1} \right)^{1/p}+ \frac{N-1}{p-1} \tan_{K,N}(t) \cos_p^{(p-1)}(\phi) \sin_p(\phi) \crcr
\phi(0)=0
\end{cases}
$$
satisfies $\phi(D/2)=\pi_p/2$, where $\pi_p, \cos_p$ and $\sin_p$ are defined as follows. 
\\For every $p \in (1,\infty)$  the positive number $\pi_p$ is defined by
$$ \pi_p:=\int_{-1}^{1} \frac{ds} {(1-s^p)^{1/p}}= \frac{2\pi}{p \sin(\pi/p)}.$$
The $C^1(\R)$ function $\sin_p: \R \to [-1,1]$ is defined implicitly on $[-\pi_p/2, 3\pi_p /2]$ by:
$$ 
\begin{cases}
t=\int_0^{\sin_p(t)} \frac{ds}{(1-s^p)^{1/p} }\quad&  \text{if } t \in \left[-\frac{\pi_p}{2}, \frac{\pi_p}{2}\right] \crcr
\sin_p(t)=\sin_p(\pi_p-t)  \quad  &\text{if } t \in \left[\frac{\pi_p}{2}, \frac{3\pi_p}{2}\right] 
\end{cases}
$$
and is periodic on $\R$. Set also by definition $\cos_p(t) = \frac{d}{dt} \sin_p(t)$. The usual fundamental trigonometric identity can be generalized by $|\sin_p(t)|^p+|\cos_p(t)|^p=1$, and so it is easily seen that $\cos_p^{(p-1)} \in C^1(\R)$. 
Clearly, if $p=2$ one finds the usual quantities:   $\pi_2=\pi, \sin_2=\sin$ and $\cos_2=\cos$.

\begin{theorem}[\cite{Matei,Val, NaVal}]\label{thm:La1pComp}
Let $K\in \R$, $N\in [1,\infty)$ and $D\in (0,\infty)$. Then the following hold
\begin{enumerate}
\item $\lambda^{1,p}_{K,N,D}=\hat{\lambda}_{K,N,D}^{1,p}$, where   $\lambda^{1,p}_{K,N,D}$ was defined in \eqref{defLa1p} and $\hat{\lambda}_{K,N,D}^{1,p}$ in  \eqref{eq:defbarla}.
\item For every fixed $p\in (1,\infty)$, the map $K,N,D \mapsto \lambda^{1,p}_{K,N,D}$ is continuous.
\item If  $K>0$ then for every $D \in (0,\pi \sqrt{N-1/K} ] $ 
$$\lambda^{1,p}_{K,N,D} \geq \lambda^{1,p}_{K,N,\pi \sqrt{N-1/K}}$$
and equality holds if and only if $D=\pi \sqrt{N-1/K} $.  If moreover $N\in \N$, then $$\lambda^{1,p}_{K,N,\pi \sqrt{N-1/K}}=\lambda^{1,p}(S^N(\sqrt{N-1/K})),$$ i.e. $\lambda^{1,p}_{K,N,\pi \sqrt{N-1/K}}$ coincides with the first eigenvalue of the $p$-laplacian on the round sphere of radius  $\sqrt{N-1/K}$.
\item If $K=0$ then $\lambda^{1,p}_{0,N,D}=(p-1) \left( \frac{\pi_p}{D}\right)^p$.
\end{enumerate}
\end{theorem}


For $K\neq 0$ and $p\neq 2$, it  is not easy to give an explicit expression of the lower bound $\lambda^{1,p}_{K,N,D}$. 
At least one can give some lower bounds, for instance recently  Li and Wang \cite{LiWang}  obtained that 
\begin{equation}
\lambda^{1,p}_{K,N,D} \geq \frac{1}{(p-1)^{p-1}} \left( \frac{NK}{N-1} \right)^{p/2} \quad \text{ for } K>0,\; p\geq 2.
\end{equation}


\subsection{$p$-spectral gap for $\CD_{loc}(K,N)$ spaces}

\begin{theorem}\label{T:spectralgap}
Let $(X,\sfd,\mm)$ be a metric measure space satisfying $\CD_{loc}(K,N)$, for some $K,N \in \R$ with $N\geq1$,  and assume moreover it is essentially non-branching. 
Let $D \in (0,\infty)$ be the diameter of $X$ and fix $p \in (1,\infty)$. Then  for any Lipschitz function $f \in L^{p}(X,\mm)$ with $\int_{X} f |f|^{p-2} \, \mm(dx) = 0$ it holds
\begin{equation}\label{eq:SpectralGap}
\lambda^{1,p}_{K,N,D} \, \int_{X} |f(x)|^{p}\, \mm(dx) \leq \int_{X} |\nabla f|^{p} (x) \, \mm(dx).
\end{equation}
In other terms it holds $\lambda^{1,p}_{(X,\sfd,\mm)}\geq \lambda^{1,p}_{K,N,D}$. Notice that for $D=\pi \sqrt{(N-1)/K}$ and $N \in \N$, it follows that $$\lambda^{1,p}_{(X,\sfd,\mm)}\geq \lambda^{1,p} (S^N((N-1)/K)).$$
\end{theorem}

\begin{proof}
Since the space $(X,\sfd)$ is bounded, then the $\CD_{loc}(K,N)$ condition implies that $\mm(X)<\infty$. Noting that the inequality  \eqref{eq:SpectralGap} is invariant under multiplication of $\mm$ by a positive constant,  we can assume without loss of generality that $\mm(X)=1$. 
Observing that   the function 
\begin{equation}\label{eq:deftildef}
\tilde{f}:= f |f|^{p-2} \in \Lip(X)
\end{equation}
 verifies the hypothesis of Theorem \ref{T:localize}, we can write $X = Y \cup \mathcal{T}$
with 
$$
\tilde{f} (x) = 0, \quad \mm\text{-a.e. }\,  y \in Y, \qquad   \mm\llcorner_{\mathcal{T}} = \int_{Q} \mm_{q}\, \qq(dq), 
$$
with $\mm_{q} = g(q,\cdot) \, \sharp \left( h_{q} \cdot \mathcal{L}^{1}\right)$, where the density $h_{q}$ verifies \eqref{E:curvdensmm} for $\qq$-a.e. $q \in Q$ and 
$$
0=\int_{X} \tilde{f}(z) \, \mm_{q}(dz) =  \int_{\dom(g(q,\cdot))} \tilde{f}(g(q,t)) \cdot h_{q}(t) \, \mathcal{L}^{1}(dt) =  \int_{\dom(g(q,\cdot))} f(g(q,t)) \, | f(g(q,t))|^{p-2}  \cdot h_{q}(t) \, \mathcal{L}^{1}(dt)
$$
for $\qq$-a.e. $q \in Q$. Now consider the map $t \mapsto f_{q} (t) : = f(g(q,t))$ and note that it is Lipschitz. 
Since ${\diam} ({\dom}(g(q,\cdot)) ) \leq D$, from the definition of $\mathcal{F}^{s}_{K,N,D}$ and of $\lambda^{1,p}_{K,N,D}$ we deduce that 
$$
\lambda^{1,p}_{K,N,D} \int_{\R} |f_{q}(t)|^{p} h_{q}(t) \, \mathcal{L}^{1}(dt) \leq \int_{\R} |f'_{q}(t)|^{p} h_{q}(t) \, \mathcal{L}^{1}(dt).
$$
Noticing that $|f'_{q}(t)| \leq |\nabla f|(g(q,t))$ one obtains that 
\begin{align*}
\lambda^{1,p}_{K,N,D} \int_{X} |f(x)|^p \, \mm (dx) &~ =  \lambda^{1,p}_{K,N,D} \int_{\mathcal{T}} |f(x)|^p \, \mm (dx) \crcr
&~ = \lambda^{1,p}_{K,N,D} \int_{Q} \left( \int_{X} |f(x)|^p \, \mm_{q} (dx) \right)\, \qq(dq) \crcr
&~ = \lambda^{1,p}_{K,N,D} \int_{Q} \left( \int_{\dom(g(q,\cdot))} |f_{q}(t)|^p \,h_{q}(t) \, \mathcal{L}^{1}(dt) \right)\, \qq(dq) \crcr
&~ \leq  \int_{Q} \left( \int_{\dom(g(q,\cdot))} |f'_{q}(t)|^{p} \,h_{q}(t) \, \mathcal{L}^{1}(dt) \right)\, \qq(dq) \crcr
&~ \leq  \int_{Q} \left( \int_{X} |\nabla f|^{p}(x) \,(g(q,\cdot))\, \sharp \left(h_{q}(t) \, \mathcal{L}^{1}\right)(dx) \right)\, \qq(dq) \crcr
&~ =  \int_{X}  |\nabla f|^{p}(x)\, \mm(dx),
\end{align*}
and the claim follows.
\end{proof}

\subsection{Almost rigidity for the $p$-spectral gap}\label{Ss:almost}
\begin{theorem}[Almost equality in the $p$-spectral gap implies almost maximal diameter]\label{T:almostrigid}
Let  $N> 1$, and $p \in (1,\infty)$ be fixed. Then for every $\ve>0$ there exists $\delta=\delta(\ve, N,p)$ such that the following holds. 

Let $(X,\sfd,\mm)$ be   an essentially non-branching metric measure space satisfying $\CD^{*}(N-1-\delta,N+\delta)$. If  $\lambda^{1,p}_{(X,\sfd,\mm)}\leq \lambda^{1,p}_{N-1,N,\pi} + \delta$, then  ${\diam}(X)\geq \pi -\ve$.
\end{theorem}

\begin{proof}
As above, without loss of generality we can assume $\mm(X)=1$.
Assume by contradiction that there exists $\ve_0>0$ such that for every $\delta>0$ we can find an  essentially non-branching metric measure space $(X,\sfd,\mm)$ satisfying $\CD^{*}(N-1-\delta,N+\delta)$, with $\mm(X)=1$, such that   ${\diam}(X)\leq \pi -\ve_0$ but  $\lambda^{1,p}_{(X,\sfd,\mm)}\leq \lambda^{1,p}_{N-1,N,\pi} + \delta$.
\\The very definition of  $\lambda^{1,p}_{(X,\sfd,\mm)}$ implies that there exists a function $f \in \Lip(X)$,  with $\int_X f |f|^{p-2} \mm=0$ and $\int_X |f|^p \, \mm(dx)=1$, such that  
\begin{equation}\label{eq:rigcontr}
\int_X |\nabla f|^p(x) \, \mm(dx) \leq \lambda^{1,p}_{(X,\sfd,\mm)}+ \delta \leq \lambda^{1,p}_{N-1,N,\pi} + 2\delta.
\end{equation}
On the other hand, Theorem \ref{thm:La1pComp} ensures that there exists $\eta>0$ such that 
$$ \lambda^{1,p}_{N-1,N,D}\geq  \lambda^{1,p}_{N-1,N,\pi} + 2\eta,  \quad \forall D \in [0,\pi-\ve_0]. $$
Moreover, the continuity of $K,N,D\mapsto \lambda^{1,p}_{K,N,D}$  guarantees that, for every $D_0\in (0,1)$ there exists $\delta_0=\delta_0(N,D_0)$ such that 
 $$\lambda^{1,p}_{N-1-\delta, N+\delta, D} \geq \lambda^{1,p}_{N-1, N, D}-\eta  \quad \forall \delta\in [0,\delta_0], \; \forall D \in [D_0, 2\pi]. $$
Since clearly by definition we have that  $\lambda^{1,p}_{K,N,D} \geq  \lambda^{1,p}_{0,N,D}$ for every $K>0, N\geq1, p \in (1,\infty)$, Theorem \ref{thm:La1pComp} gives  that
 $$\lim_{D\downarrow 0} \lambda^{1,p}_{N-1-\delta,N+\delta,D}  \geq  \lim_{D\downarrow 0} \lambda^{1,p}_{0,N+\delta,D} = +\infty$$ 
 uniformly for $\delta \in [0, \delta_0(N)]$. The combination of the last two estimates  yields
\begin{equation}
\lambda^{1,p}_{N-1-\delta, N+\delta, D} \geq  \lambda^{1,p}_{N-1,N,\pi} + \eta \quad \forall D \in [0,\pi-\ve_0], \; \forall \delta \in [0, \delta_0(N)].
\end{equation}  
By repeating the proof of Theorem \ref{T:spectralgap}, and observing that by construction it holds $\diam(\dom(g(q,\cdot)) \leq \pi-\ve_0$,  we then obtain
\begin{align*}
\int_{X}  |\nabla f|^{p}(x)\, \mm(dx)&~ = \int_{Q} \left( \int_{X} |\nabla f|^{p}(x) \,(g(q,\cdot))\, \sharp \left(h_{q}(t) \, \mathcal{L}^{1}\right)(dx) \right)\, \qq(dq) \crcr
&~ \geq   \int_{Q} \left( \int_{\dom(g(q,\cdot))} |f'_{q}(t)|^{p} \,h_{q}(t) \, \mathcal{L}^{1}(dt) \right)\, \qq(dq) \crcr
&~ \geq \int_{Q}   \lambda^{1,p}_{N-1-\delta,N+\delta,\diam(\dom(g(q,\cdot))}  \left( \int_{\dom(g(q,\cdot))}   |f_{q}(t)|^p \,h_{q}(t) \, \mathcal{L}^{1}(dt) \right)\, \qq(dq) \crcr
&~ \geq   (\lambda^{1,p}_{N-1,N,\pi} + \eta)     \int_{Q}    \left( \int_{\dom(g(q,\cdot))}   |f_{q}(t)|^p \,h_{q}(t) \, \mathcal{L}^{1}(dt) \right)\, \qq(dq) \crcr
&~ =  \lambda^{1,p}_{N-1,N,\pi} + \eta.
\end{align*}
Contradicting  \eqref{eq:rigcontr},   once chosen $\delta<\eta/2$.
\end{proof}

\begin{corollary}[Almost equality in the $p$-spectral gap implies mGH-closeness to a spherical suspension]\label{cor:pSpecDiam}
Let $N\geq 2$, and $p \in (1,\infty)$ be fixed. Then for every $\ve>0$ there exists $\delta=\delta(\ve, N,p)>0$ such that the following holds. 

Let $(X,\sfd,\mm)$ be   an $\RCD^* (N-1-\delta,N+\delta)$ space. If  
$$\lambda^{1,p}_{(X,\sfd,\mm)}\leq \lambda^{1,p}_{N-1,N,\pi} + \delta,$$
then there exists an $\RCD^*(N-2,N-1)$ space $(Y, \sfd_Y, \mm_Y)$  such that 
$$\sfd_{mGH}\left( (X,\sfd,\mm),  [0,\pi] \times_{\sin}^{N-1}  Y \right) \leq \ve.  $$
\end{corollary}

\begin{proof}
Fix $N\in [2, \infty), p \in (1,\infty) $ and assume by contradiction there exist $\ve_0>0$ and  a sequence $(X_j, \sfd_j, \mm_j)$ of $\RCD^*(N-1-\frac{1}{j}, N+\frac{1}{j})$ spaces such that  $\lambda^{1,p}_{(X,\sfd,\mm)}\leq \lambda^{1,p}_{N-1,N,\pi} + \frac{1}{j},$  but
\begin{equation}\label{eq:contrj}
\sfd_{mGH}(X_j, [0,\pi] \times_{\sin}^{N-1} Y) \geq \ve_0 \quad \text{for every $j\in \N$}
\end{equation}
and every  $\RCD^*(N-2,N-1)$ space $(Y, \sfd_Y, \mm_Y)$ with $\mm_Y(Y)=1$.   Observe that Theorem  \ref{T:almostrigid} yields 
\begin{equation}\label{eq:diam}
\diam((X_j, \sfd_j))\to \pi.
\end{equation}
 By the compactness/stability property of $\RCD^*(K,N)$ spaces recalled in Theorem \ref{thm:CompRCD}  we get  that, up to subsequences, the spaces $X_j$ mGH-converge to a limit $\RCD^*(N-1,N)$ space $(X_\infty, \sfd_\infty, \mm_\infty)$.   Since the diameter is continuous under mGH convergence of uniformly bounded spaces,    \eqref{eq:diam} implies  that $\diam((X_\infty, \sfd_\infty))=\pi$.  But then by the Maximal Diameter Theorem \cite{Ket} we get that $(X_\infty, \sfd_\infty, \mm_\infty)$ is isomorphic to a spherical suspension  $[0,\pi] \times_{\sin}^{N-1} Y$ for some  $\RCD^*(N-2,N-1)$ space $(Y, \sfd_Y, \mm_Y)$ with $\mm_Y(Y)=1$.  Clearly this contradicts \eqref{eq:contrj} and the thesis follows.
\end{proof}

\begin{corollary}[$p$-Obata Theorem]\label{cor:pObata}
Let $(X,\sfd,\mm)$ be   an $\RCD^* (N-1,N)$ space for some $N \geq 2,$ and let $1<p<\infty$. If  
$$\lambda^{1,p}_{(X,\sfd,\mm)}= \lambda^{1,p}_{N-1,N,\pi} \quad (=\lambda^{1,p}(S^N), \text{ if }N\in \N),$$
then $(X,\sfd,\mm)$ is a spherical suspension, i.e. there exists an $\RCD^*(N-2,N-1)$ space $(Y, \sfd_Y, \mm_Y)$  such that $(X,\sfd,\mm)$ is isomorphic to  $[0,\pi] \times_{\sin}^{N-1}  Y$.
\end{corollary}

\begin{proof}
Theorem   \ref{T:almostrigid} implies that $\diam((X, \sfd))=\pi$ and the thesis then follows by the  Maximal Diameter Theorem \cite{Ket}.
\end{proof}

\begin{remark}
The Obata's Theorem for $p=2$ in $\RCD^*(N-1,N)$ spaces has been recently obtained by Ketterer  \cite{Ket} by different methods (see also \cite{JZ1}); the approach proposed here has the double advantage of length and of being valid for every $p\in (1,\infty)$.
\end{remark}

\section{The case $p=1$ and the Cheeger constant}\label{S:p=1}
It is well known (see for  instance \cite{Honda, WWZ}) that  an alternative way of 
defining  $\lambda^{1,p}_{(X,\sfd,\mm)}$ which extends also to $p=1$ is the following. 
For every $p\in [1, \infty)$ and every $f \in L^p(X)$ let 
$$
c_p(f):=\inf_{c \in \R} \left( \int_X |f-c|^p \, \mm \right)^{1/p}.
$$
For every $p \in (1,\infty)$ it holds that \cite[Corollary 2.11]{Honda}
$$
\lambda^{1,p}_{(X,\sfd,\mm)}=\inf \left\{ \int_X |\nabla f|^p \, \mm\; \text{:} \; f \in \Lip\cap L^p(X), \, c_p(f)=\|f\|_{L^p}=1  \right\} .
$$
It is then natural to set
\begin{equation}\label{eq:defla11}
\lambda^{1,1}_{(X,\sfd,\mm)}=\inf \left\{ \int_X |\nabla f| \, \mm\; \text{:} \; f \in \Lip\cap L^1(X), \, c_1(f)=\|f\|_{L^1}=1  \right\}. 
\end{equation}
Assuming that $\mm(X)=1$, recall that a number $M_f \in \R$ is a median for $f$  if and only if
$$
\mm(\{f\geq M_f \}) \geq \frac{1}{2} \qquad \text{and} \qquad  \mm(\{f\leq M_f \}) \geq \frac{1}{2}.
$$
It is not difficult to check that (see for instance \cite[Section VI]{Chav}) for every  $f \in L^1(X)$ there exists a median of $f$, and moreover 
$$
\int_X |f-M_f| \, \mm=c_1(f)
$$
holds for every median $M_f$ of $f$. This link between  $c_1(f)$ and $M_f$ is useful to prove the equivalence between 
the Cheeger constant and $\lambda^{1,1}_{(X,\sfd,\mm)}$. Recall that the Cheeger constant $h_{(X,\sfd,\mm)}$ is defined by
$$
h_{(X,\sfd,\mm)}:=\inf\left\{\frac{\mm^+(E)}{\mm(E)} \; : \; E\subset X \text{ is Borel  and } \mm(E)\in (0, 1/2]   \right\},
$$
where 
$$
\mm^+(E):= \liminf_{\ve \downarrow 0} \frac{\mm(E^\ve)- \mm(E)}{\ve}
$$
is the (outer) Minkowski content. 
As usual  $E^{\ve}:=\{x \in X \,:\, \exists y \in  E \, \text{ such that } \, \sfd(x,y)< \ve \}$ is the $\ve$-neighborhood of $E$ with respect to the metric $\sfd$.
The next result, due to Maz'ya \cite{Maz} and Federer-Fleming \cite{FeFl} (see also \cite{Bob} for a careful derivation, \cite[Lemma 2.2]{Mil09} 
and \cite[Proposition 2.13]{Honda} for the present formulation), rewrites Cheeger's
isoperimetric inequality in functional form.

\begin{proposition}
Assume that $(X,\sfd,\mm)$ is a m.m.s with $\mm(\{x\})=0$ for every $x \in X$, i.e. $\mm$ is atomless. Then 
$$h_{(X,\sfd,\mm)}=\lambda^{1,1}_{(X,\sfd,\mm)}.$$
\end{proposition}

It is then clear that the comparison and almost rigidity theorems for $\lambda^{1,1}$ will be based on the corresponding isoperimetric ones obtained by the authors in \cite{CM1}. To this aim in the next subsection we briefly recall  the model Cheeger constant  for the comparison.

\subsection{The model Cheeger constant $h_{K,N,D}$}\label{SS:hKND}
If $K>0$ and $N\in \N$, by the L\'evy-Gromov isoperimetric inequality we know that, for $N$-dimensional smooth manifolds having Ricci curvature bounded below by $K$, the Cheeger constant i is bounded below by the one of the $N$-dimensional round sphere of the suitable radius. In other words  the \emph{model} Cheeger constant is the one of ${\mathbb S}^N$. For $N\geq 1, K\in \R$ arbitrary real numbers the situation is  more complicated, and just recently E. Milman \cite{Mil} discovered what is the model Cheeger constant (more precisely he discovered the model isoperimetric profile, which in turn implies the  model Cheeger constant). In this short section we recall its definition.
\\

Given $\delta>0$, set 
\[
\begin{array}{ccc}
 s_\delta(t) := \begin{cases}
\sin(\sqrt{\delta} t)/\sqrt{\delta} & \delta > 0 \\
t & \delta = 0 \\
\sinh(\sqrt{-\delta} t)/\sqrt{-\delta} & \delta < 0
\end{cases}

& , &

 c_\delta(t) := \begin{cases}
\cos(\sqrt{\delta} t) & \delta > 0 \\
1 & \delta = 0 \\
\cosh(\sqrt{-\delta} t) & \delta < 0
\end{cases}
\end{array} ~.
\]
Given a continuous function $f :\R \to  \R$ with $f(0) \geq 0$, we denote by $f_+ : \R \to \R^+ $ the function coinciding with $f$ between its first non-positive and first positive roots, and vanishing everywhere else, i.e. $f_+ := f \chi_{[\xi_{-},\xi_{+}]}$ with $\xi_{-} = \sup\{\xi \leq 0; f(\xi) = 0\}$ and $\xi_{+} = \inf\{\xi > 0; f(\xi) = 0\}$.

Given $H,K \in \R$ and $N \in [1,\infty)$, set $\delta := K / (N-1)$  and define the following (Jacobian) function of $t \in \R$:
\[
J_{H,K,N}(t) :=
\begin{cases}
\chi_{\{t=0\}}  & N = 1 , K > 0 \\
\chi_{\{H t \geq 0}\} & N = 1 , K \leq 0 \\
\left(c_\delta(t) + \frac{H}{N-1} s_\delta(t)\right)_+^{N-1} & N \in (1,\infty) \\
\end{cases} ~.
\]
As last piece of notation, given a non-negative integrable function $f$ on a closed interval $L \subset \R$, we denote with $\mu_{f,L}$  
the probability measure supported in $L$ with density (with respect to the Lebesgue measure) proportional to $f$ there. In order to simplify a bit the notation we will write
$h_{(L,f)}$ in place of $h_{(L,\, |\cdot|, \mu_{f,L})}$.
\\The model Cheeger constant for spaces having Ricci curvature bounded below by $K\in \R$, dimension bounded above by $N\geq 1$ and diameter at most $D\in (0,\infty]$ is then defined by
\begin{equation}\label{eq:defIKND}
h_{K,N,D}:=\inf_{H\in \R,a\in [0,D]} h_{\left([-a,D-a], J_{H,K,N}\right)}. 
\end{equation}
The formula above has the advantage of considering all the possible cases in just one equation, 
but  probably it is  also instructive to  isolate the different cases in a more explicit way. Indeed one can check \cite[Section 4] {Mil} that:
\begin{itemize}

\item \textbf{Case 1}: $K>0$ and $D<\sqrt{\frac{N-1}{K}} \pi$,
$$
h_{K,N,D}=   \inf_{\xi \in \big[0, \sqrt{\frac{N-1}{K}} \pi -D\big]} h_{\big( [\xi,\xi +D],  \sin( \sqrt{\frac{K}{N-1}} t)^{N-1} \big)}. 
$$

\item \textbf{Case 2}:   $K > 0$ and $D \geq \sqrt{\frac{N-1}{K}} \pi$, 
$$
h_{K,N,D} = h_{ \big( [0, \sqrt{\frac{N-1}{K}} \pi],  \sin( \sqrt{\frac{K}{N-1}} t)^{N-1}   \big)}. 
$$

\item  \textbf{Case 3}: $K=0$ and $D<\infty$,
\begin{eqnarray*}
h_{K,N,D}&=&  \min \left \{ \begin{array}{l}  \inf_{\xi \geq 0} h_{([\xi,\xi+D],  t^{N-1})}~,\\
\phantom{\inf_{\xi \in \R}} h_{([0,D],1)}
\end{array}
\right \} \\
&=& \frac{N}{D} \inf_{\xi \geq 0, v\in (0, 1/2)}  \frac{\left(\min(v,1-v) (\xi+1)^{N} + \max(v,1-v) \xi^{N}\right)^{\frac{N-1}{N}}}{ v \left[ (\xi+1)^{N} - \xi^{N} \right]}  . 
\end{eqnarray*}

\item \textbf{Case 4}:  $K < 0$, $D<\infty$:
$$
h_{K,N,D}=  \min \left \{ \begin{array}{l}
\inf_{\xi \geq 0} h_{\big([\xi,\xi+D], \; \sinh(\sqrt{\frac{-K}{N-1}} t)^{N-1}\big)} ~ , \\
\phantom{\inf_{\xi \in \R}} h_{\big([0,D],   \exp(\sqrt{-K (N-1)} t) \big)}  ~, \\
\inf_{\xi \in \R} h_{\big([\xi,\xi+D], \; \cosh(\sqrt{\frac{-K}{N-1}} t)^{N-1}\big)}
\end{array}
\right \} .
$$
\item In all the remaining cases,  the model Cheeger constant trivializes: $h_{K,N.D}=0$.  
\end{itemize}
\medskip

\subsection{Sharp comparison and almost rigidity for $\lambda^{1,1}=h$}
\begin{theorem}
Let $(X,\sfd,\mm)$ be an  essentially non-branching $\CD_{loc}(K,N)$-space  for some $K\in \R, N \in [1,\infty)$,  with  $\mm(X)=1$ and  having diameter  $D\in (0,+\infty]$. Then
\begin{equation}\label{eq:hcomp}
h_{(X,\sfd,\mm)}\geq h_{K,N,D}.
\end{equation}  
Moreover, for $K>0$ the following holds: for every $N>1$ and $\ve>0$ there exists $\bar{\delta}=\bar{\delta}(K,N,\ve)$ such that,  for every $\delta\in [0,\bar{\delta}]$, if $(X,\sfd,\mm)$ is an essentially non-branching $\CD^{*}(K-\delta,N+\delta)$-space such that 
\begin{equation}\label{eq:hdelta}
h_{(X,\sfd,\mm)}\leq h_{K,N,\pi\sqrt{(N-1)/K}}+\delta\quad (=h(S^N(\sqrt{(N-1)/K)})+\delta, \text{ if } N \in \N), 
\end{equation}
then  $\diam(X)\geq \pi\sqrt{(N-1)/K}-\ve$.
\end{theorem}

\begin{proof}
Recall that the isoperimetric profile of $(X,\sfd,\mm)$ is the largest function $\cI_{(X,\sfd,\mm)}:[0,1]\to \R^+$ such that for every Borel subset $E\subset X$ it holds $\mm^+(E)\geq \cI_{(X,\sfd,\mm)}(\mm(E))$. As discovered  in \cite{Mil} (see also \cite[Section 2.5]{CM1} for the present notation), for every $K\in \R, N\in [1, \infty)$, $D\in(0,\infty]$ there exists a model isoperimetric profile $\cI_{K,N,D}:[0,1]\to \R^+$;  it is straightforward to check that  
$$
h_{(X,\sfd,\mm)}=\inf_{v \in (0,1/2)} \frac{\cI_{(X,\sfd,\mm)}(v)}{ v} \quad \text{and} \quad h_{K,N,D}=\inf_{v \in (0,1/2)} \frac{\cI_{K,N,D}(v)}{ v}. 
$$
Since in our previous paper \cite[Theorem 1.2]{CM1} we proved that for every $v>0$ it holds 
\begin{equation}\label{eq:compI}
\cI_{(X,\sfd,\mm)}(v) \geq \cI_{K,N,D}(v),
\end{equation}
 the first claim \eqref{eq:hcomp} follows.
\\In order to prove the second part of the theorem, note  \eqref{eq:hdelta} implies that there exists $\bar{v}\in(0,1/2)$ such that
\begin{equation} \nonumber \label{eq:cIXKd}
\frac{\cI_{(X,\sfd,\mm)}(\bar{v}) }{\bar{v}}\leq h_{(X,\sfd,\mm)}+\delta \leq  h_{K,N,\pi\sqrt{(N-1)/K}}+2\delta \leq \frac{\cI_{K,N,\pi\sqrt{(N-1)/K}} (\bar{v})}{\bar{v}}+2 \delta.
\end{equation}
Multiplying by $\bar{v}$, we get 
$$
\cI_{(X,\sfd,\mm)}(\bar{v}) \leq  \cI_{K,N,\pi\sqrt{(N-1)/K}} (\bar{v}) +2 \delta \bar{v} \leq  \cI_{K,N,\pi\sqrt{(N-1)/K}} (\bar{v}) + \delta.
$$ 
The thesis then follows by direct  application of  \cite[Theorem 1.5]{CM1}.

\end{proof}

Before stating the  result let us observe that if $(X,\sfd,\mm)$ is an $\RCD^*(K,N)$ space for some $K>0$ then, 
called $\sfd':=\sqrt{\frac{K}{N-1}} \; \sfd$, we have that $(X,\sfd',\mm)$ is  $\RCD^*(N-1,N)$; 
in other words, if the Ricci lower bound is $K>0$ then up to scaling we can assume it is actually equal to $N-1$.

Arguing as in the proof of Corollaries \ref{cor:pSpecDiam}-\ref{cor:pObata} we get the following result.

\begin{corollary}
For every $N\in [2, \infty) $,  $\ve>0$ there exists $\bar{\delta}=\bar{\delta}(N,\ve)>0$ such that the following hold. For every  $\delta \in [0, \bar{\delta}]$, if  $(X,\sfd,\mm)$ is an $\RCD^*(N-1-\delta,N+\delta)$-space with $\mm(X)=1$, satisfying 
$$h_{(X,\sfd,\mm)}\leq h_{N-1,N,\pi}+\delta\quad (=h(S^N)+\delta,  \text{ if } N \in \N), $$
then  there exists an $\RCD^*(N-2,N-1)$ space $(Y, \sfd_Y, \mm_Y)$ with $\mm_Y(Y)=1$ such that 
$$\sfd_{mGH}(X, [0,\pi] \times_{\sin}^{N-1} Y) \leq \ve. $$
In particular, if $(X,\sfd,\mm)$ is an $\RCD^*(N-1,N)$-space satisfying $h_{(X,\sfd,\mm)}= h_{N-1,N,\pi}=h(S^N)$, then it is isomorphic to a spherical suspension; i.e. there exists an $\RCD^*(N-2,N-1)$ space $(Y, \sfd_Y, \mm_Y)$ with $\mm_Y(Y)=1$ such that $(X,\sfd,\mm)$ is isomorphic to $[0,\pi] \times_{\sin}^{N-1} Y$.
\end{corollary}


\section{Sharp Log-Sobolev and Talagrand inequalities}

\subsection{Sharp Log-Sobolev in diameter-curvature-dimensional form}\label{SS:Log-Sob}

Recall that a m.m.s. $(X,\sfd,\mm)$ supports the Log-Sobolev inequality with constant $\alpha>0$ 
if for any Lipschitz function $f:X\to [0,\infty)$ with $\int_X f\,  \mm=1$ it holds
\begin{equation}\label{eq:logsobolev}
2\alpha \int_X f \log f \, \mm \leq \int_{\{f>0\}} \frac{|\nabla f|^2}{f} \mm.
\end{equation}
The largest constant $\alpha$, such that \eqref{eq:logsobolev} holds for any  Lipschitz function $f:X\to [0,\infty)$ with $\int_X f\,  \mm=1$, 
will be called Log-Sobolev constant of $(X,\sfd,\mm)$ and denoted with $\alpha^{LS}_{(X,\sfd,\mm)}$.

As before  we will reduce to the one-dimensional case. 
Given $K\in \R$, $N\geq1$, $D\in (0,+\infty]$ we denote with $\alpha^{LS}_{K,N,D}>0$ the maximal constant $\alpha$ such that 
\begin{equation}\label{eq:logsobolevKND}
2\,\alpha \int_\R f \log f \, \mu \leq \int_{\{f>0\}} \frac{|f'|^2}{f} \mu, \quad \forall \mu\in {\mathcal F}^s_{K,N,D},
\end{equation}
for every Lipschitz $f:\R\to [0,\infty)$  with  $\int f \, \mu=1$.  

\begin{remark}
If $K>0$ and $D=\pi \sqrt{\frac{N-1}{K} }$, it is known that the corresponding optimal 
Log-Sobolev constant is $\frac{KN}{N-1}$ (see the discussion below). 
It is an interesting open problem, that we don't address here, 
to give an explicit expression of the quantity  $\alpha^{LS}_{K,N,D}$ for general $K\in \R, N\geq 1$, $D \in (0, \infty)$. 
\end{remark}

\begin{theorem}\label{thm:compLS}
Let $(X,\sfd,\mm)$ be a metric measure space with diameter $D \in (0,\infty)$ and satisfying $\CD_{loc}(K,N)$ for some $K\in \R, N\in [1,\infty)$.  
Assume moreover it is essentially non-branching. 
Then  for any Lipschitz function $f:X\to [0,\infty)$ with $\int_X f\,  \mm=1$ it holds
$$
2\,\alpha^{LS}_{K,N,D} \int_X f \log f \, \mm \leq \int_{\{f>0\}} \frac{|\nabla f|^2}{f} \mm.
$$
In other terms it holds $\alpha^{LS}_{(X,\sfd,\mm)}\geq \alpha^{LS}_{K,N,D}$.
\end{theorem}

\begin{proof}
Since $\CD_{loc}(K,N)$ implies that the measure is locally doubling, the finiteness of the diameter implies that $\mm(X)<\infty$.  Observing that the Log-Sobolev inequality  \eqref{eq:logsobolev} is invariant under a multiplication of $\mm$ by a constant, 
we can then assume without loss of generality that $\mm(X)=1$. 
Consider any Lipschitz function with $\int_X f\,  \mm=1$ and apply Theorem \ref{T:localize} to $\hat f: =  1 - f$. 
Hence we can write $X = Y \cup \mathcal{T}$ with 
$$
f (y) = 1, \quad \mm\text{-a.e. }\,  y \in Y, \qquad   \mm\llcorner_{\mathcal{T}} = \int_{Q} \mm_{q}\, \qq(dq), 
$$
with $\mm_{q} = g(q,\cdot) \, \sharp \left( h_{q} \cdot \mathcal{L}^{1}\right)$, the density $h_{q}$ verifies \eqref{E:curvdensmm} for $\qq$-a.e. $q \in Q$ and 
$$
 1 = \int_{X} f(z) \, \mm_{q}(dz) =  \int_{\dom(g(q,\cdot))} f(g(q,t)) \cdot h_{q}(t) \, \mathcal{L}^{1}(dt) 
$$
for $\qq$-a.e. $q \in Q$. Now consider the map $t \mapsto f_{q} (t) : = f(g(q,t))$ and note that it is Lipschitz. 
Since ${\diam} ({\dom}(g(q,\cdot)) ) \leq D$, from the definition of $\mathcal{F}^{s}_{K,N,D}$ and of $\alpha^{LS}_{K,N,D}$ we deduce that 
$$
2\alpha^{LS}_{K,N,D} \int_{\R}  f_{q}(t) \log(f_{q}(t)) \,  h_{q}(t) \, \mathcal{L}^{1}(dt) \leq \int_{\{f_q(\cdot)>0\}} \frac{|f'_{q}(t)|^2}{f_{q}(t)} \, h_{q}(t) \, \mathcal{L}^{1}(dt).
$$
Noticing that $|f'_{q}(t)| \leq |\nabla f|(g(q,t))$ and that $f \log f$ vanishes over $Y$, one obtains that 
\begin{align*}
2 \alpha^{LS}_{K,N,D} \int_{X}  f \log f \, \mm(dx)  &~ = 2 \alpha^{LS}_{K,N,D} \int_{\mathcal{T}}  f \log f \, \mm(dx) \crcr
		&~ 	= 	2 \alpha^{LS}_{K,N,D} \int_{Q} \left( \int_{X}   f \log f    \, \mm_{q} (dx) \right)\, \qq(dq) \crcr
		&~ 	=  	2 \alpha^{LS}_{K,N,D}   \int_{Q} \left( \int_{\dom(g(q,\cdot))}    f_{q}(t) \log(f_{q}(t))  \,h_{q}(t) \, \mathcal{L}^{1}(dt) \right)\, \qq(dq) \crcr
		&~ 	\leq  	\int_{Q} \left( \int_{\dom(g(q,\cdot)) \cap \{f_q(\cdot)>0\}}  \frac{|f'_{q}(t)|^2}{f_{q}(t)}  \,h_{q}(t) \, \mathcal{L}^{1}(dt) \right)\, \qq(dq) \crcr
		&~ 	\leq  	\int_{Q} \left( \int_{\{f>0\}}  \frac{|\nabla f|^2}{f} \; (g(q,\cdot))\, \sharp \left(h_{q}(t) \, \mathcal{L}^{1}\right)(dx) \right)\, \qq(dq) \crcr
		&~ 	\leq  	\int_{\{f>0\}}  \frac{|\nabla f|^2}{f} \, \mm(dx),
\end{align*}
and the claim follows.
\end{proof}

If $K>0$, by Bonnet-Myers diameter bound, we know that if $(X,\sfd,\mm)$ satisfies $\CD_{loc}(K,N)$ then $\diam(X)\leq \pi \sqrt{\frac{N-1}{K}}$. 
Recalling definition \eqref{eq:logsobolevKND}, we then set $\alpha^{LS}_{K,N}:=\alpha^{LS}_{K,N, \pi \sqrt{\frac{N-1}{K}}}$
for the Log-Sobolev constant without an upper diameter bound. 
By applying the regularization of the measures $h\, {\mathcal L}^1 \in {\mathcal F}^s_{K,N, \pi \sqrt{\frac{N-1}{K}}}$ 
discussed in Lemma \ref{lem:approxh} and arguing analogously to the proof of Proposition \ref{P:sLa=La}, 
we get that in the definition of $\alpha^{LS}_{K,N}$ it is equivalent to take the $\inf$ among measures 
in ${\mathcal F}_{K,N, \pi \sqrt{\frac{N-1}{K}}}$,  defined in \eqref{eq:DiffIne}.
But now  if $\mu\in {\mathcal F}_{K,N, \pi \sqrt{\frac{N-1}{K}}}$ is a probability measure on  $\R$ with smooth density satisfying the $\CD_{loc}(K,N)$ 
condition for $K>0$, it is known that the sharp Log-Sobolev constant is $\alpha^{LS}_{K,N}=\frac{K N}{N-1}$ (see for instance  \cite[Proposition 6.6]{Bakry}). 
More precisely, as proved by Mueller-Weissler \cite{MueWei}, for  every $K>0$ and $N\geq 1$, the sharp constant is attained by  
the usual model  probability measure on the interval $\big[0, \sqrt{\frac{N-1}{K}} \pi \big]$ 
proportional to $\sin\big(\sqrt{ \frac{K}{N-1}} t \big)^{N-1}$; notice that for $N \in \N$ it corresponds 
to the round sphere of radius $\sqrt{\frac{N-1}{K}}$. We then have the following corollary.

\begin{corollary}[Sharp Log-Sobolev under $\CD^{*}(K,N), K>0, N> 1$]\label{cor:LSKN}
Let $(X,\sfd,\mm)$ be a metric measure space satisfying $\CD^{*}(K,N)$ for some $K>0, N> 1$,  and assume moreover it is essentially non-branching.   
Then  for any Lipschitz function $f:X\to [0,\infty)$ with $\int_X f\,  \mm=1$ it holds
$$
\frac{2KN }{N-1} \int_X f \log f \, \mm \leq \int_{\{f>0\}} \frac{|\nabla f|^2}{f} \mm.
$$
In other terms it holds $\alpha^{LS}_{(X,\sfd,\mm)}\geq \frac{K N}{N-1}$.
\end{corollary}

Let us mention that, since the reduction to a 1-D problem is done via an $L^1$-optimal transportation argument, 
Corollary \ref{cor:LSKN} can be seen as  a solution to  \cite[Open Problem 21.6]{Vil}.

\subsection{From Sharp Log-Sobolev to Sharp  Talagrand}
First of all let us recall that the relative entropy functional $Ent_\mm: {\mathcal P}(X)\to [0, +\infty]$ with respect to a given $\mm\in {\mathcal P}(X)$ is defined to be 
$$
Ent_\mm(\mu)= \int_X \varrho\, \log \varrho \, \mm, \quad \text{ if } \mu=\varrho \mm 
 $$
and $+\infty$ otherwise.

Otto-Villani \cite{OtVil} proved that for smooth Riemannian manifolds the Log-Sobolev inequality with constant $\alpha>0$ 
implies the Talagrand inequality with constant $\frac{2}{\alpha}$ preserving sharpness. 
The result was then generalized to arbitrary metric measure spaces by Gigli-Ledoux \cite{GL}, so that we can state:

\begin{theorem}[From  Log-Sobolev to  Talagrand, \cite{OtVil, GL}]\label{thm:LST}
Let $(X,\sfd,\mm)$ be a metric measure space supporting the Log-Sobolev inequality with constant $\alpha>0$. 
Then it also supports the Talagrand inequality with constant $\frac{2}{\alpha}$, i.e. it holds
$$
W_2^2(\mu, \mm) \leq \frac{2}{\alpha} Ent_\mm(\mu)
$$
for all $\mu \in \mathcal{P}(X)$.
\end{theorem}

Combining Theorem \ref{thm:compLS} with Theorem \ref{thm:LST} we get Theorem \ref{thm:talagrand} which improves the 
Talagrand constant $2/K$, which is sharp for $\CD(K,\infty)$ spaces, by a factor $N-1/N$ in case the dimension is bounded above by $N$. 
This constant is sharp for $\CD_{loc}(K,N)$ spaces, indeed it is sharp already in the smooth setting \cite[Remark 22.43]{Vil}.  
Since both our proof of the sharp Log-Sobolev inequality and the proof of Theorem \ref{thm:LST}  
are essentially optimal transport based,  this be seen as an answer to \cite[Open Problem 22.44]{Vil}.

\begin{remark}[Sharpness and estimates of the best constants]
Recall that for weighted smooth manifolds,  the Log-Sobolev inequality implies the Talagrand inequality which in turns implies  
the Poincar\'e inequality every step without any loss in the constants \cite[Theorem 22.17]{Vil}.  
Since when we compute  the comparison Log-Sobolev constant $\alpha^{LS}_{K,N,D}$ and  
the comparison first eigenvalue $\lambda^{1,2}_{K,N,D}$, we work with the smooth measures ${\mathcal F}_{K,N,D}$ on $\R$, we always have the estimate 
\begin{equation}\label{eq:LSPoi}
\alpha^{LS}_{K,N,D} \leq  \lambda^{1,2}_{K,N,D}.
\end{equation}
Notice that, for $K>0$ and $D=\sqrt{ \frac{N-1}{K}} \pi$ they actually coincide:
\begin{equation}
\frac{KN}{N-1} = \alpha^{LS}_{K,N,\sqrt{ \frac{N-1}{K}} \pi} = \lambda^{1,2}_{K,N,\sqrt{ \frac{N-1}{K}} \pi}.  
\end{equation}
An interesting question we do not address here is if this is always the case, 
i.e. if  in \eqref{eq:LSPoi} equality holds for every $K\in \R$, $N \geq 1$, $D\in (0,\infty)$. 
Since the value of  $\lambda^{1,2}_{K,N,D}$ is known in many cases, it would have as a consequence 
the determination of the explicit value of the best constant in both the Log-Sobolev and the Talagrand inequalities in the curvature-dimension-diameter forms. 
This would also imply  rigidity and almost-rigidity statements attached to the Log-Sobolev and Talagrand inequalities, as proven here for the Poincar\'e inequality. 
Let us note that for the almost rigidity to hold for both the Log-Sobolev and Talagrand inequalities it would be enough 
to prove that for every $\varepsilon>0$ there exists $\delta>0$ such that 
$\alpha^{LS}_{K,N,D}\geq  \alpha^{LS}_{K,N, \sqrt{\frac{N-1}{K}} \pi} + \delta=\frac{KN}{N-1} +\delta$,  if $D\in \big[0, \sqrt{\frac{N-1}{K}} \ve- \delta \big]$.
\end{remark}


\section{Sharp Sobolev Inequalities}

Recall that $(X,\sfd,\mm)$ supports a $(p,q)$-Sobolev inequality with constant $\alpha^{p,q}$ if for any  Lipschitz function  $f : X \to \R$ it holds
\begin{equation}\label{E:Sobolev}
\frac{\alpha^{p,q}}{p-q} \left\{ \left( \int_{X}  |f |^{p} \, \mm \right)^{\frac{q}{p}} -   \int_{X}  |f |^{q} \, \mm \right\} \leq   \int_{X}  |\nabla f |^{q} \, \mm.
\end{equation}
The largest constant $\alpha^{p,q}$ such that \eqref{E:Sobolev} holds for any Lipschitz function $f$ will be called the $(p,q)$-Sobolev constant 
of $(X,\sfd, \mm)$ and will be denoted by $\alpha^{p,q}_{(X,\sfd,\mm)}$. 

Again we consider the one-dimensional case and given $K \in \R, N\geq 1$ and $D \in (0,\infty]$ we define 
$^s\alpha^{p,q}_{K,N,D}$ to be the maximal constant $\alpha$
such that 
$$
\frac{\alpha}{p-q} \left\{ \left( \int_{X}  |f |^{p} \, \mu \right)^{\frac{q}{p}} -   \int_{X}  |f |^{q} \, \mu \right\} \leq   \int_{X}  |\nabla f |^{q} \, \mu, 
\qquad \forall\ \mu \in \mathcal{F}^{s}_{K,N,D},
$$
for every Lipschitz function $f:\R \to \R$. Restricting the maximization to $\mu \in \mathcal{F}_{K,N,D}$, we obtain the constant 
$\alpha^{p,q}_{K,N,D}$. Using the approximation Lemma \ref{lem:approxh} and reasoning as in Proposition \ref{P:sLa=La} one obtains that 
$$
^{s}\alpha^{p,q}_{K,N,D} = \alpha^{p,q}_{K,N,D}. 
$$

\begin{theorem}\label{thm:compL}
Let $(X,\sfd,\mm)$ be a metric measure space with  diameter $D \in (0,\infty)$ and satisfying $\CD_{loc}(K,N)$ for some $K\in \R, N\in [1,\infty)$.  
Assume moreover it is essentially non-branching. 
Then  for any Lipschitz function it holds
$$
\frac{\alpha^{p,q}_{K,N,D}}{p-q} \left\{ \left( \int_{X}  |f (x)|^{p} \, \mm(dx) \right)^{\frac{q}{p}} - \int_{X}  |f (x)|^{q} \, \mm(dx)   \right\} 
\leq    \int_{X}  |\nabla f (x)|^{q} \, \mm(dx), 
$$
In other terms, it holds $\alpha^{p,q}_{(X,\sfd,\mm)} \geq \alpha^{p,q}_{K,N,D}$.
\end{theorem}

\begin{proof}
First of all note that  $\CD_{loc}(K,N)$  coupled with the finiteness of the diameter implies  $\mm(X)<\infty$.

{\bf Step 1: The case $p> q$.} 

\noindent
With a slight abuse of notation $q$ will denote both the exponent in the Sobolev embedding and the index in the disintegration, there  should be no confusion since the clearly different roles.
Fix any Lipschitz function $f$ and consider the function $\hat f (x): = 1 - c |f(x)|^{p}$, with $c : = 1/ (\int |f|^{p} \mm)$. Therefore $\int \hat f \, \mm =0$ and we can invoke 
Theorem \ref{T:localize}.  Hence $X = Y \cup \mathcal{T}$ with 
$$
\hat f (y) = 0, \quad \mm\text{-a.e. }\,  y \in Y, \qquad   \mm\llcorner_{\mathcal{T}} = \int_{Q} \mm_{q}\, \qq(dq), 
$$
with $\mm_{q} = g(q,\cdot) \, \sharp \left( h_{q} \cdot \mathcal{L}^{1}\right)$, the density $h_{q}$ verifies \eqref{E:curvdensmm} for $\qq$-a.e. $q \in Q$ and 
$$
0 = \int_{X} \hat f(z) \, \mm_{q}(dz) =  \int_{\dom(g(q,\cdot))} \hat f(g(q,t)) \cdot h_{q}(t) \, \mathcal{L}^{1}(dt) 
$$
for $\qq$-a.e. $q \in Q$. 

Now consider the map $t \mapsto f_{q} (t) : = f(g(q,t))$ and note that it is Lipschitz. 
Since ${\diam} ({\dom}(g(q,\cdot)) ) \leq D$, from the definition of $\mathcal{F}^{s}_{K,N,D}$ and of $\alpha^{p,q}_{K,N,D}$ we deduce that 
$$
\left( \int_{\R}  |f_{q} (t)|^{p} h_{q}(t) \, \mathcal{L}^{1}(dt) \right)^{\frac{q}{p}} \leq  \int_{\R}  | f_{q} (t)|^{q}h_{q}(t) \, \mathcal{L}^{1}(dt)  
+ \frac{p-q}{\alpha^{p,q}_{K,N,D}} \int_{\R}  | f' (t)|^{q} h_{q}(t) \, \mathcal{L}^{1}(dt).
$$
Since for $\qq$-a.e. $q\in Q$ it holds $\int \hat f \, \mm_{q} = 0$, it follows that
$$
\int_{X} |f(x)|^{p}\,\mm_{q}(dx) = \frac{1}{c} = \int_{X} |f(x)|^{p}\,\mm(dx).
$$
Therefore the previous inequality reads as 
$$
1 \leq \left( \frac{1}{\int |f(x)|^{p}\,\mm(dx)} \right)^{\frac{q}{p}} \left( \int_{X}  | f_{q}|^{q} \, \mm_{q}
+ \frac{p-q}{\alpha^{p,q}_{K,N,D}} \int_{X}  | f' |^{q}  \, \mm_{q} \right).
$$
Noticing that $|f'_{q}(t)| \leq |\nabla f|(g(q,t))$, integrating over $Q$ one obtains that 
\begin{equation}\label{eq:mTSob}
\mm(\mathcal{T}) \leq  \left( \frac{1}{\int |f(x)|^{p}\,\mm(dx)} \right)^{\frac{q}{p}} \, \int_{\mathcal{T}}  | f (x)|^{q} \, \mm(dx)
+ \frac{p-q}{\alpha^{p,q}_{K,N,D}} \int_{\mathcal{T}}  | \nabla f (x)|^{q}  \, \mm(dx).
\end{equation}
To complete the argument one should prove that for each $y \in Y$
$$
1 \leq \left( \frac{1}{\int |f|^{p}\,\mm} \right)^{\frac{q}{p}} \left(  | f(y)|^{q} 
+ \frac{p-q}{\alpha^{p,q}_{K,N,D}}   | \nabla f (y)|^{q}  \right).
$$
As for $\mm$-a.e. $y\in Y$ one has $|f(y)|^{p} =  \int_{X} |f|^{p}\,\mm$, this last inequality holds trivially.  Integrating this last inequality over $Y$ and adding it to \eqref{eq:mTSob}, we obtain  the claim.
\medskip

{\bf Step 2: The case $p <q$.} It follows repeating the previous localization argument and writing the Sobolev inequality in the following form
$$
\left( \int_{X}  |f (x)|^{p}  \, \mm(dx) \right)^{\frac{q}{p}} \geq  \int_{X}  | f (x)|^{q} \, \mm(dx)  
- \frac{q-p}{\alpha } \int_{X}  | \nabla f (x)|^{q}  \, \mm(dx).
$$
\end{proof}

As already observed, if $K > 0$ then $\diam (X) \leq \pi \sqrt{(N-1)/K}$ and therefore one can define 
$$
\alpha^{p,q}_{K,N} : = \alpha^{p,q}_{K,N, \pi \sqrt{(N-1)/K}}, 
$$
the $(p,q)$-Sobolev inequality with no diameter upper bound. 
If $\mu\in {\mathcal F}_{K,N, \pi \sqrt{(N-1)/K}}$ with $K>0$, 
it is known that the sharp $(p,2)$-Sobolev constant, verifies (see for instance  \cite[Theorem 3.1]{Led-Toul})
$$
\alpha^{p,2}_{K,N} \geq \frac{K N}{ N-1 }, \qquad \textrm{for }\ 1\leq p\leq \frac{2N}{N -2}.
$$
Moreover, for $N \in \N$, it is attained on the round sphere of radius $\sqrt{\frac{N-1}{K}}$. We then have the following corollary.

\begin{corollary}\label{cor:pqSKN}
Let $(X,\sfd,\mm)$ be a metric measure space satisfying $\CD^{*}(K,N)$ for some $K>0, N\in (2,\infty)$,  and assume moreover it is essentially non-branching.   
Then  for any Lipschitz function $f$ it holds
$$
\frac{KN }{(p-2)(N-1)}  \left\{ \left( \int_{X}  |f |^{p} \, \mm \right)^{\frac{2}{p}} -   \int_{X}  |f |^{2} \, \mm \right\} \leq   \int_{X}  |\nabla f |^{2} \, \mm, 
$$
for any $2 < p \leq 2N/(N -2)$. In other terms it holds $\alpha^{p,2}_{(X,\sfd,\mm)}\geq \frac{K N}{N-1}$.
\end{corollary}

Corollary \ref{cor:pqSKN} can be seen as  a solution to  \cite[Open Problem 21.11]{Vil}.

\section*{Appendix}

All the inequalities we have presented here rely on the general scheme of applying one-dimensional localization to a big family of inequalities, 
called 4-functions inequalities (see for instance the work of Kannan-Lov\'asz-Simonovits \cite{KaLoSi}). 

The argument goes as follows. Suppose we are interested in proving that for $f_{1},f_{2},f_{3},f_{4}$ integrable functions and $\alpha,\beta>0$ it holds
\begin{equation}\label{Eq:4function}
\left( \int_{X} f_{1} \, \mm \right)^{\alpha} \left( \int_{X} f_{2} \, \mm \right)^{\beta} \leq \left( \int_{X} f_{3} \, \mm \right)^{\alpha} \left( \int_{X} f_{4} \, \mm \right)^{\beta}.
\end{equation}
Then consider the one-dimensional localization induced by $g : = f_{3} - c f_{1}$, with $c = (\int f_{3} \mm)/(\int f_{1} \mm)$:
$$
\mm\llcorner_{\mathcal{T}} = \int_{Q} \mm_{q} \qq(dq),
$$
where $X = \mathcal{T} \cup Y$ and on $Y$ it holds $g(x) =0$ for $\mm$-a.e. $x\in Y$.
Then it is sufficient to prove that
\begin{align*}
\left( \int_{X} f_{1} \, \mm_{q} \right)^{\alpha} \left( \int_{X} f_{2} \, \mm_{q} \right)^{\beta} 
& \leq  \left( \int_{X} f_{3} \, \mm_{q} \right)^{\alpha} \left( \int_{X} f_{4} \, \mm_{q} \right)^{\beta}, \quad \qq-\textrm{a.e.} \ q \in Q \crcr
f_{2} (x)  
& \leq  c^{\alpha/\beta}  f_{4} (x), \qquad \mm-\textrm{a.e.} \ x \in Y.
\end{align*}
Indeed from the localization it follows that $\int g \, \mm_{q} = 0$ for $\qq$-a.e. $q \in Q$ and therefore 
$$
\int_{X} f_{2} (x)  \, \mm_{q}(dx)  \leq  c^{\alpha/\beta}  \int_{X} f_{4} (x) \, \mm_{q}(dq), \qquad \qq-\textrm{a.e.} \ q \in Q. 
$$
Integrating over $Q$ and adding the integral over $Y$, \eqref{Eq:4function} follows.

\end{document}